\documentclass[a4paper,11pt]{amsart}
\usepackage{a4wide,graphicx,color}
\usepackage{amsmath}
\usepackage{amsfonts}
\usepackage{amsthm}
\usepackage{amssymb}
\usepackage[active]{srcltx}
\usepackage[all]{xy}

\usepackage{subfigure}
\usepackage{setspace}
\usepackage{mdwlist}

\usepackage{hyperref}
\usepackage{enumerate}
\usepackage{mathrsfs}
\usepackage{mdwlist}

\usepackage{textcomp}
\usepackage{wasysym}

\newcommand{\R}{\mathbb{R}}

\newtheorem{theorem}{Theorem}
\newtheorem{lemma}{Lemma}[section]
\newtheorem{definition}{Definition}[section]

\newtheorem{corollary}{Corollary}[section]
\newtheorem{prop}{Proposition}[section]
\theoremstyle{remark}
\newtheorem{remark}{Remark}[section]

\author[A. Ciomaga]{Adina Ciomaga$^\dag$}
\address{$^\dag$Centre de Math\'ematiques et de Leurs
  Applications, Ecole Normale Sup\'erieure de Cachan, CNRS, UniverSud,
  61 avenue du pr\'esident Wilson, F-94230 Cachan, France}
\email{ciomaga@cmla.ens-cachan.fr}

\author[J.M. Morel]{Jean-Michel Morel$^\ddag$}
\address{$^\ddag$Centre de Math\'ematiques et de Leurs
  Applications, Ecole Normale Sup\'erieure de Cachan, CNRS, UniverSud,
  61 avenue du pr\'esident Wilson, F-94230 Cachan, France}
\email{morel@cmla.ens-cachan.fr}

\title[Level Lines Shortening]
{A proof of equivalence between level lines shortening and curvature motion in image processing}

\graphicspath{{figures/}}

\begin{document}

\begin{abstract}  
In this paper we define the continuous  Level Lines Shortening evolution of a two-dimensional image as the Curve Shortening operator acting simultaneously and independently on all the level lines of the initial data, and show that it computes  a viscosity solution for the mean curvature motion. This provides an exact analytical framework for its numerical implementation, which runs on line on any image at http://www.ipol.im/. Analogous results hold for its affine variant version, the Level Lines Affine Shortening.
\end{abstract}

\maketitle

{\bf Keywords:} 
partial differential equations, 
mean curvature motion, 
affine curvature motion, 
curve shortening, 
affine shortening, 
level lines, 
topographic maps 
\medskip

{\bf  AMS Subject Classification:} 
35K65, 53A04, 53C21, 65N99, 65D18
\bigskip

\bigskip\section{Introduction}
\bigskip

In \cite{CMM10},  \cite{CMM11} there were introduced two new image processing algorithms, \emph{Level Lines Shortening} and its affine variant \emph{Level Lines Affine Shortening} simulating a contrast invariant and scale invariant evolution of an image by mean curvature motion, respectively affine curvature motion. The aim of this work is to rigorously justify that the Level Lines Shortening algorithm proposed in these  works computes explicitly a viscosity solution for the \emph{Mean Curvature Motion}
\begin{equation}\label{eq_MCM}
 \frac{\partial u}{\partial t}=|Du|curv(u),
\end{equation}
respectively that the Level Lines Affine Shortening algorithm provides a viscosity solution for  the \emph{Affine Curvature Motion}
\begin{equation}\label{eq_ACM}
 \frac{\partial u}{\partial t}=|Du|\big(curv(u)\big)^{1/3}.
\end{equation}
These equations are particularly interesting since they can be axiomatically obtained from the image multiscale theories, as  the unique partial differential equations satisfying the most desired invariance properties in computer vision. This axiomatic characterization was given by  Alvarez, Guichard, Lions and Morel in \cite{AGLM93}.  Caselles et al. \cite{CaCoMo:96:KP} realized the potential of processing directly the image level lines. They proposed to perform a contrast invariant image analysis directly on the set of level lines, or {\it topographic map}. A fast algorithm computing the topographic map was developed by Monasse and Guichard in \cite{MG98}.

The Level Lines Shortening numerical chain stands for the simultaneous and independent curvature evolution of \emph{all} the level lines for a given function. Curve smoothing by the intrinsic heat equation, also called \emph{Curve Shortening}

\begin{equation}\label{eq_CS}
\frac{\partial x}{\partial t}={\bf k}(x).
\end{equation}
was one of the first versions of curve analysis proposed by Mackworth and Mokhtarian in \cite{MM92}. Here ${\bf k}(x)$ denotes the curvature vector at $x$, defined as the second derivative $x''(s)$ with respect to any length parameter $s$.  By this (nonlinear) evolution a curve instantly becomes smooth, shrinks asymptotically to a circle and develops no singularities or self-crossings. Rigorous proofs were given by Gage and Hamilton for convex Jordan curves \cite{GH86} and later extended to embedded curves by Grayson \cite{Gr87}. The \emph{Affine Shortening} equation

\begin{equation}\label{eq_AS}
\frac{\partial x}{\partial t}= \left( |{\bf k}|^{-2/3} {\bf k}\right)(x)
\end{equation}
is a surprising variant of curve shortening introduced by Sapiro and Tannenbaum in \cite{SaTa:93:AISS},  \cite{SaTa:93:ICE}. Angenent, Sapiro and Tannenbaum \cite{SaTa:94:ACE} gave the existence and uniqueness proofs for affine shortening  and showed a result similar to Grayson's theorem: a shape eventually becomes convex and thereafter evolves towards an ellipse before collapsing. A remarkably fast and geometric algorithm for affine shortening  was given by Moisan in \cite{Mo98}.

On the other hand, Osher-Sethian defined and studied in \cite{OS88} the level set method for the motion of fronts by (mean) curvature, for which  Chen-Giga-Goto  \cite{CGG91} and Evans-Spruck \cite{ES91} provided rigorous justifications. In this setting, the initial curve $\Sigma_0$ is considered as the zero level set of some function $u_0$. The basic result asserts that the zero level set of the evolved function $$\Sigma_t = \{u(\cdot,t)=0\}$$ does not depend on the choice of the initial data and therefore the evolution is purely geometrical. Their arguments are based on the notion of viscosity solution of Crandall and Lions, that allows one to give a suitable meaning to the (MCM) and (ACM) equations, in the class of uniformly continuous functions. We refer to the 'user's guide' of Crandall, Ishii and Lions \cite{CIL92} for further details about viscosity solutions. A mathematical  link  between the median filter and the motion by mean curvature was conjectured by  Merriman, Bence and Osher \cite{MBO92} and later proved by Barles and Georgelin in \cite{BG95} using viscosity methods and by Evans in \cite{Ev93} using a nonlinear semigroup approach.

Evans and Spruck checked in \cite{ES91} the consistency of the level set approach with the classical motion by mean curvature. More precisely, they showed that the mean curvature motion agrees with the classical motion, if and as long as the latter exists. The results applies for a smooth hypersurface, given as the connected boundary of a bounded open set. Their arguments are based on comparison techniques with lower barriers for the approximated mean curvature motion and strongly use the fact that the hypersurface is the zero level set of its (signed) distance function. However, the result does not describe the complete behavior of \emph{all} the level lines of a Lipschitz function. Namely, if we are given a Lipschitz function $u_0$ which evolves by mean curvature in the viscosity sense, are all of its level lines evolving independently by curve shortening? 
For dimension $n\geq 3$ the result is not true, since hypersurfaces can develop singularities and possible change topology. Thanks to Grayson's theorem, the 2D case has a very peculiar structure which we will take advantage of and show that evolving independently and simultaneously by curve shortening \emph{all} the level lines of a function is equivalent to applying directly a mean curvature motion to the functions itself.

The \emph{Level Lines Shortening} builds on the above mentioned contributions and connects explicitly the geometric approach for curve shortening evolutions and the viscosity framework for curvature motions. More precisely, this operator first extracts all the level lines of an image, then it independently and simultaneously smooths {\rm all} of its level lines by curve shortening (CS) (respectively affine shortening (AS)) and eventually reconstructs, at each step, a new image from the evolved level lines. The chains are based on a topological structure, the inclusion tree of level lines as a full and non-redundant representation of an image \cite{CM10}, and on a topological property,  the monotonicity  of curve shortening with respect to inclusion. Therefore, the hierarchy of the level lines is  maintained while performing the smoothing.

We show in this paper that the image reconstructed from the evolved level lines is a viscosity solution of the mean curvature motion (MCM) (resp. affine curvature motion (ACM)). The initial image will be considered as an element of a particular space of functions $\mathcal{VS}(\Omega)$ that we term space of {\rm very simple} functions. This corresponds to bilinearly interpolated images defined on a rectangle $\Omega$ whose topographic maps contain only Jordan curves.
The set of very simple functions arises naturally in image processing, since level lines corresponding to noncritical levels are sufficient to grant an exact reconstruction of the digital image. 

{In this way, the described algorithm corresponds exactly to its numerical implementation \cite{IPOL} and has the advantage of satisfying both numerically and analytically all the invariance properties required by the scale space in question.

The paper is organized as follows. In section \S 2 we define the class of very simple functions as approximations for Lipschitz functions. Section \S 3 is devoted to the definition of the Level Lines Shortening evolution, as an operator acting both on crowns of Jordan curves and flat areas, and on very simple functions. In the last section \S 4 we give the equivalence result. 
The result is then extended to general Lipschitz functions, by first approximating the initial data by very simple functions and then using standard stability properties of viscosity solutions.

\bigskip\section{Modeling Level Lines Shortening}

\bigskip\subsection{Crowns of Jordan Curves}

A Jordan curve is a one to one continuous map from the unit circle $S^1$ into $\R^2$.
A Jordan curve $\Sigma$ splits the plane in two connected components. We denote by $Int(\Sigma)$ the open bounded component and by $Ext(\Sigma)$ the open unbounded component.

\begin{definition}[Partial order]
Let $\Sigma^1$ and $\Sigma^2$ be two Jordan curves.
We say that $\Sigma^{1}$ \textbf{surrounds} (strictly) $\Sigma^{2}$ and we write
$\Sigma^{1}\preceq \Sigma^{2}$ ($\Sigma^{1}\prec \Sigma^{2}$)
if $ Int(\Sigma^1)\subseteq Int(\Sigma^2)$ (respectively $\overline{Int(\Sigma^1)}\subset Int(\Sigma^2)$).
\end{definition}
This defines a partial order on the set of planar Jordan curves.

\begin{definition}\label{def::Lip}
We say that a Jordan $\Sigma$ curve is \textbf{piecewise $C^1$}, or in $C_p^{0,1}(S^1)$ if it has finite length
$l(\Sigma)\geq 0$, any length parametrization is piecewise $C^1$  and
if at each discontinuity point for the tangent there are left and right tangent vectors, which are not collinear.
\end{definition}

\begin{definition}\label{def::topoLip}
We say that a sequence of curves $\Sigma^n$ \textbf{converges} in $C_p^{0,1}(S^1)$ to a curve $\Sigma$ if, denoting by $s\in [0, l(\Sigma)]\to x(s)$ a length parameterization of $\Sigma$,  there are Lipschitz parameterizations $s\in[0,l(\Sigma)]\to x^n(s)$ for $\Sigma^n$ such that $x^n(s)$ tends uniformly to $x(s)$, and the left and right unit tangent vectors of $x^n(s)$ tend uniformly to the left and right tangent vectors of $x(s)$. 
\end{definition}

\begin{remark}{\rm
When $l(\Sigma)=0$, these conditions are reduced to the uniform convergence of $x^n(s)$ toward $x(s)$. This convergence can be defined by a family of neighborhoods around each element of $C_p^{0,1}(S^1)$, which therefore is a  topological space.}
\end{remark}

\begin{definition}\label{def::crown}
A \textbf{crown} $\Sigma: (\lambda,\mu)\to C_p^{0,1}(S^1)$ is a continuous and monotone
map $\Sigma$ from the interval $(\lambda, \mu)$ into $ C_p^{0,1}(S^1)$ 
endowed with the partial order $\preceq$.  
If the map is defined on the closed/open interval we talk about \textbf{closed/open crown}. 
\end{definition}

\begin{definition}
When the crown is closed and increasing, $\Sigma(\lambda)=\Sigma^\lambda$ is called the \textbf{interior curve of the  crown} and $\Sigma(\mu)=\Sigma^\mu$ its \textbf{exterior curve of the crown}. If the crown is decreasing, these names exchange. The range of the open/closed crown is denoted by $\Sigma(]\lambda,\mu[)$,  $\Sigma([\lambda,\mu])$, respectively. The crown itself  as an ordered family of curves will also be denoted by  $\Sigma^{[\lambda,\mu]}$
or $(\Sigma^\nu)_{\nu\in(\lambda,\mu)}$.
\end{definition}

\begin{remark}{\rm
A(closed)  crown is a homeomorphism from a closed interval $[\lambda,\mu]$ to $\Sigma[\lambda,\mu]$, since the map $\Sigma$ is continuous, and the interval $[\lambda,\mu]$ is compact.
}\end{remark}

\begin{definition}
Let $A \subset\Omega$ be a bounded, connected set, whose boundary consists of a finite number of disjoint Jordan curves. We call \textbf{exterior curve of } $A$ the unique Jordan curve $\Sigma^e$ whose interior contains $A$ and \textbf{interior curve(s)} the other Jordan curves $\Sigma^i$, $i\in I$.
\end{definition}

\subsection{The class of Very Simple Functions}

A digital image is usually  known by its samples
$\left\{u(i,j)\right\}_{0\leq i\leq M, 0\leq j\leq N}$ on a rectangular grid of $\Omega=[0,M]\times [0,N]$.
We assume that the underlying image $0\leq u(x)\leq 1$ whose samples are the $u(i,j)$ is a Lipschitz function defined on $\Omega$, the continuous image domain. We shall always assume that $u(x)=0$ on the boundary of the domain $\partial \Omega$ and that $u(x)>0$ in the interior of $\Omega$.

The bilinear interpolation in $\Omega$ is the simplest continuous interpolation from the discrete samples $u(i,j)$.  This  interpolate, still denoted $u$, is defined as the unique function coinciding with  the digital image $u$ on the samples
which is bilinear in each dual pixel (the square formed by the centers of four adjacent pixels). This means that $u$ has the form
$$ u(x_1,x_2)=\alpha x_1+\beta x_2+\gamma x_1x_2 +\delta $$
on each square with vertices   $(i, j)$, $(i+1,j)$, $(i, j+1)$, $(i+1, j+1)$. This bilinear interpolation is therefore
positive on the interior of the domain and zero on $\partial \Omega$. The set of bilinear interpolates of  digital images on $\Omega$ will be denoted by $\mathcal{BL}(\Omega).$
The next result is a sane consistency property of the bilinear interpolation.
\begin{prop} \label{approximationuniforme}
If $u$ defined on $\Omega$ is Lipschitz and only known by its samples, its bilinear interpolate converges uniformly to $u$ when the grid mesh tends to zero. 
\end{prop}

The references  \cite{LMR01}, \cite{LMMM00}, \cite{CM10} show that the bilinear interpolation brings a long list of numerically useful topological properties.

\begin{prop}
\emph{(Properties of level lines of bilinear interpolates)}
\begin{itemize*}
\item[(L1)] For  every level $0\leq \lambda\leq 1$ except for a finite set of levels $\lambda_1,\dots, \lambda_n$ called \textbf{critical}, the iso-level set $\{u=\lambda\}$ is the disjoint union of a finite set of piecewise-$C^1$
Jordan curves, denoted by $(\Sigma^{\lambda, i})_{i\in I_\lambda}$ where $I_\lambda$ is a finite set of indices;
\item[(L2)] The open set $\Omega\setminus u^{-1}(\{\lambda_1,\dots,\lambda_n\}) $ has a finite number of connected components. Each connected component is the range of an open crown $\Sigma^{]\mu, \nu[}$  where $\mu, \nu\in\{\lambda_1,\dots,\lambda_n\}$. As a consequence, $\Omega$ is partitioned in open crowns and in the closed iso-level sets $\{u=\lambda_i\}$ corresponding to the critical levels.
\end{itemize*}
 \end{prop}

\begin{proof}[Sketch of proof]
A dual pixel will contain a critical level either if it is flat, or if it contains a saddle point. In both cases there is only one critical level in the pixel. Since there is a finite number of pixels, there is a finite number of critical levels. At any other level, the restriction of an iso-level set to a given pixel is either empty, or is a single piece of hyperbola. Since the bilinear interpolate is continuous, these pieces of hyperbolae concatenate at each pixel boundary to form one or several disjoint Jordan curves.
\end{proof}

\medskip
\begin{figure}
\centering
\subfigure[]{\includegraphics[width=0.32\columnwidth]{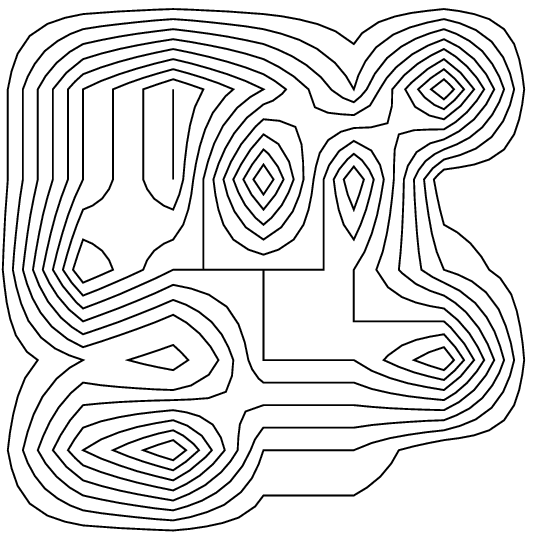}}
\subfigure[]{\includegraphics[width=0.32\columnwidth]{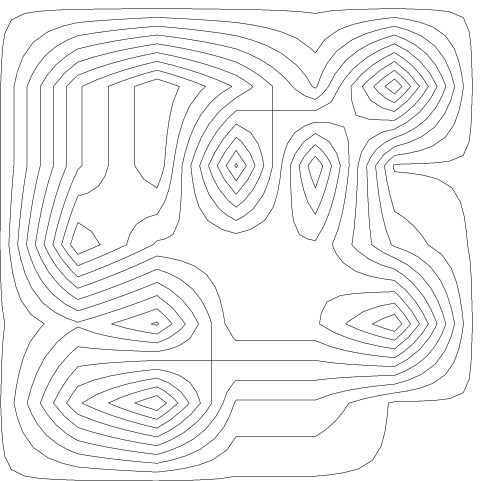}}
\subfigure[]{\includegraphics[width=0.32\columnwidth]{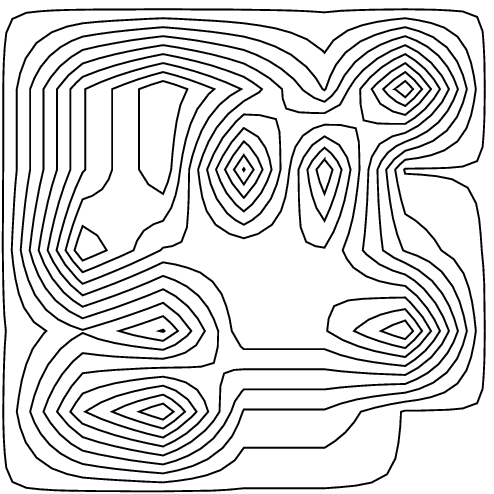}}
\caption{\small Topological structure of the family of level lines, for a bilinear interpolated image.
For  every level $0\leq \lambda\leq 1$, except for a finite set, the iso-level sets $\{u=\lambda\}$ are made of a finite set of piecewise $C^1$ Jordan curves. At critical levels $\lambda_k$, they can have a rather complicated form: they can have T-junctions or reduce to segments (a), contain square pixels (thus being flat areas) or have self-crossings at saddle points (b).
However, by fattening the iso-level sets corresponding to the critical levels, the topology of level sets becomes very simple: only Jordan curves or flat sets (c).\label{fig:pixbil}  }
\end{figure}

The above structure is quite simple, but it does not describe the structure of level lines at the critical levels $\lambda_k$, which can take a rather complicated  form as illustrated in Figure \ref{fig:pixbil}. To define a curvature evolution for these sets would be cumbersome. We shall overcome this drawback, both numerically and analytically, by building a still simpler approximation. This approximation of the image is obtained by fattening all critical iso-level sets into  open sets whose boundary is  a finite set of Jordan level curves. Their curvature evolution will simply be defined by the curvature evolution of their boundary.
Thus, we define a still simpler approximation for Lipschitz functions, deduced from the bilinear interpolation.

\begin{definition}
We say that a Lipschitz function $u$ on $\Omega$ is \textbf{ very simple} and we denote by $u\in\mathcal{VS}(\Omega)$ if it satisfies properties $(L1)$, $(L2)$ and
\begin{itemize*}
\item[(L3)] Each crown $\Sigma^{ ]\mu,\nu[}$ can be completed into a closed crown by adding its interior and exterior curves $\Sigma^{\mu}$ and $\Sigma^{\nu}$, which are limits in the $C_p^{0,1}(S^1)$ topology of the level lines of the crown.
\end{itemize*}
\end{definition}\label{defverysimplefunction}

\begin{remark}{\rm 
The construction of the partition goes as follows.
The domain $\Omega$ is the union of all its isolevel sets $\{u=\lambda\}$. 
By the property $(L1)$,  these sets are disjoint unions of (smooth) Jordan curves, except for a finite number. Property $(L2)$ further tell us that these Jordan curves are organized in a finite number of open crowns. At this level, $\Omega$ is the union between a finite number of open crowns and \emph{the} finite number of iso-level sets leftover (which in turn can be split in a finite number of connected components, since we lay on a compact). The third property $(L3)$ tells us that the boundaries of these flat areas, which are not necessarily Jordan curves,  can be shifted to the crowns, such that the crowns become closed.
As a consequence, for very simple functions, $\Omega$ is the disjoint union of the ranges of a finite number of compact crowns, and of a finite number of flat open connected components, each belonging to some critical level $\lambda_k$,
\begin{equation}\label{eq:partition}
\Omega  = \left(\bigcup_{\substack{k \in K}}
\Sigma^{ [\mu_k,\nu_k],k}
\right)\cup\left(\bigcup_{\substack{j\in\{1,\cdots, n\},\\ l\in J_j}} F^{\lambda_j,l}\right)
\end{equation} where $K$ is the finite set of all crown indexes, $\mu_k,\nu_k\in\{\lambda_1,\cdots,\lambda_n\}$ are critical levels, and $J_j$ is the finite set of indexes of  open connected components of the iso-level set of $u$ at the critical level $\lambda_j$.}
\end{remark}

\begin{definition}
We call \textbf{flat regions} of $u_0\in\mathcal{VS}(\Omega)$ the open sets $F^{\lambda_j,l}$, whose boundaries are  unions of a finite number of piecewise-$C^1$ Jordan level curves, and on which the function is constant:
$$ u(x) = \lambda_j, \ \forall x \in F^{\lambda_j, l}, \; l\in J_j. $$
\end{definition}

\begin{lemma} \label{lemma:approxVS}
Every Lipschitz function  can be approximated uniformly by a sequence of very simple functions.
\end{lemma}

\begin{proof}
For all $u \in Lip(\Omega)$ and $ \varepsilon>0 $ we build a function
$v\in \mathcal{VS}(\Omega)$ such that
$$\sup_{x\in \Omega}|u(x)-v(x)| <\varepsilon.$$
By Proposition \ref{approximationuniforme}, we can assume that $u\in \mathcal{BL}(\Omega)$. We can order the set of its critical levels  $\lambda_1< \lambda_2 < ... < \lambda_n$. Then define, for $0<\varepsilon<\frac 12\min_j|\lambda_{j+1}-\lambda_j|$, the function  $v(x) = f_\varepsilon(u(x))$, where $f$ is the 1-Lipschitz nondecreasing  function
\begin{equation}
 f_\varepsilon (t) = \left\{
 	\begin{array}{lll}
  		t, & t\leq\lambda_1- \varepsilon & \\
 		 \lambda_k - k\varepsilon & \lambda_k- \varepsilon \leq t\leq \lambda_k + \varepsilon, & k=\overline{1,n}\\
  		t- k\varepsilon & \lambda_k+ \varepsilon \leq t\leq \lambda_{k+1}-\varepsilon, & k=\overline{1,n-1}\\
  		t-n\varepsilon, & t\geq \lambda_n+ \varepsilon &
 	\end{array} \right.
\end{equation}
We have $|f(t)-t|\leq n\varepsilon$, thus $\max_{\Omega}|u-v|\leq n\varepsilon$  which proves the convergence claim.

It remains to check that the function $v\in\mathcal{VS}(\Omega)$. The critical levels of $v$ are inherited from  $u$ and consist of $\{\lambda_k - k \varepsilon\}_{k=1,...,n}$. By construction, all the level lines of $v$ at noncritical levels are level lines of $u$ at noncritical levels and hence Jordan level curves, grouped in crowns. Consider now the flat regions of $v$, which are the open connected components of the  sets
$$\{\lambda_k-\varepsilon< u<\lambda_{k}+\varepsilon\}, \ k=\overline{1,n}.$$
Since $\varepsilon<\frac 12\min_k(\lambda_{k+1}-\lambda_k),$ the boundary of each one of these flat regions is contained in the union of the level sets $\{u=\lambda_k-\varepsilon\}$ and $u=\{\lambda_k+\varepsilon\}$ which are a finite number of Jordan curves. Each one of these Jordan curves belongs to a crown $\Sigma^{ ]\lambda_k,\lambda_{k'}[}$ of $u$ which is truncated into a crown $\Sigma^{]\lambda_k -k\varepsilon,\lambda_{k'}-k'\varepsilon[}$ of $v$.
 \end{proof}

\bigskip\section{Level Lines  Shortening operator}

\medskip\subsection{LLS semigroup operator}
Given a very simple function $u_0\in\mathcal{VS}(\Omega)$, its Level Lines  Shortening evolution consists in evolving independently and simultaneously by Curve  Shortening each of its level lines, denoted by $\Sigma_0^{\lambda, i}$, and eventually reconstructing, for each time $t>0$, a new function $u(\cdot,t)$ whose level lines are the evolutions  $\Sigma_t^{\lambda, i}$, where the subscript $t$ denotes the time $t$. This definition, will have to be proven consistent. Our goal is therefore to prove  the commutative diagram:
$$
\xymatrix{
 u_0(\cdot) \ar@{.>}[d]_{MCM/LLS}
        \ar[rrrr]^{level\space\ lines\space\ extraction}
 &&&&\{\Sigma_{0}^{\lambda,i}\}_{\lambda,i}
        \ar[d]^{CS}\\
 u(\cdot,t)	
&&&& \{\Sigma_{t}^{\lambda,i}\}_{\lambda,i}.
        \ar[llll]_{reconstruction}}
$$
To this end, we shall use several fine properties of curve shortening evolution \cite{Gr87}, \cite{GH86}, which is given
 in terms of a nonlinear geometric partial differential equation
\begin{equation}\label{eq::CS}
\frac{\partial x}{\partial t}(s,t)  = {\bf k}(s,t) \\
\end{equation}
where $x(s,t)$ is a family of smooth Jordan curves parameterized for each $t$ by a length parameter. The vector ${\bf k}$ is the acceleration which is normal to the curve, points towards its concavity, and whose norm is  the inverse of the radius of the osculating circle.

\begin{theorem} The  curve shortening evolution has the following properties:
\begin{enumerate}

\item[($P1$)] For any Jordan $C^{0,1}(S^1)$ curve $\Sigma_0$, there exists a collapsing time $T(\Sigma_0)>0$ such  that the Cauchy problem (\ref{eq::CS}) has a unique solution $\Sigma_t$ in
$$
C^{0,1}\left(S^1\times\left[0,T(\Sigma_0)\right)\right)\cap C^{\infty}\left(S^1\times\left(0,T(\Sigma_0)\right)\right)
$$
which still is a Jordan curve.  We set $\Sigma_t=\emptyset$ for  $t>T(\Sigma_0)$.\medskip

\item[($P2$)] The  map $\Sigma_0 \mapsto \Sigma_t$ is continuous for the $C_p^{0,1}(S^1)$ topology of Definition \ref{def::Lip}.
 For  $t=T(\Sigma_0)$ the curve collapses to a point $x(\Sigma_0)$. \medskip

\item[($P3$)] Before collapsing at time $T(\Sigma_0)$, the curve $\Sigma_t$ - rescaled at constant area equal to $\pi$ -
 converges in the $C_p^{0,1}(S^1)$ topology to the unit circle centered at the collapsing point $x(\Sigma_0)$.\medskip

\item[($P4$)] Inclusion Principle:
\begin{itemize}
 \item if $\Sigma^{1}\preceq\Sigma^{2}$, then $\Sigma_t^{1}\preceq \Sigma_t^{2}$  for all $t>0$.
 \item if $\Sigma^{1}\prec\Sigma^{2}$, then $\Sigma_t^{1}\prec \Sigma_t^{2}$  for all $t>0$.
\end{itemize}\medskip

\item[($P5$)] The min-distance of any two disjoint curves increases with time until one of the curve collapses
$$
dist(\Sigma_s^1,\Sigma_s^2) < dist(\Sigma_t^1,\Sigma_t^2), \forall s\leq t.
$$
\item[($P6$)] Convex curves remain convex and shrink in time: $\Sigma_t\preceq \Sigma_0$.
\end{enumerate}\label{propsofLLSoncurves}
\end{theorem}

\begin{theorem}
Properties $(P1)-(P6)$ hold for affine shortening:
\begin{equation}\label{eq::AS}
\frac{\partial x}{\partial t}(s,t)  = |{\bf k}|^{-2/3}{\bf k}(s,t) \\
\end{equation}
except for $(P3)$ which is replaced by
\begin{itemize}
 \item[$(P3)'$] Before collapsing at time $T(\Sigma_0)$, the curve $\Sigma_t$ 
 converges in the $C_p^{0,1}(S^1)$ topology to an ellipse centered at the collapsing point $x(\Sigma_0)$.\medskip
\end{itemize}\label{propsofLLASoncurves}
\end{theorem}

Let  $u_0\in\mathcal{VS}(\Omega)$ be a very simple function and $\{\Sigma_0^{\lambda,i}\}_{\lambda,i\in I_\lambda}$ its level lines. Denote by $\Sigma_t^{\lambda,i}$ the evolution of $\Sigma_0^{\lambda,i}$ at time $t$
$$ \Sigma_0^{\lambda,i} \xrightarrow{CS}\Sigma_t^{\lambda,i}. $$
Our first purpose is to show that the family of smooth Jordan curves $\{\Sigma_t^{\lambda,i}\}_{\lambda,i\in I_\lambda}$ is actually the set of level lines of a very simple image $u(\cdot,t)$.
\medskip

\begin{definition} Let $\Sigma_0^{[\zeta,\mu]}= \left(\Sigma_0^\lambda\right)_{\lambda\in [\zeta,\mu]}$ be a closed crown.
We call \textbf{level lines shortening of the crown} $\Sigma_0^{[\zeta,\mu]}$ the family of curves
$$ LLS(t)\left(\Sigma_0^{[\zeta,\mu]}\right) := (\Sigma_t^\lambda)_{\lambda\in [\zeta,\mu]}. $$
where $\Sigma_t^\lambda$ are the curve (affine) shortening evolutions of $\Sigma_0^\lambda$ for all $\lambda\in [\zeta,\mu]$.
\end{definition}
To fix ideas we refer in the following to increasing crowns. Analogous results hold for decreasing crowns.
\begin{prop}\label{evolutionofcrown}
Consider a closed increasing crown $\Sigma^{[\zeta, \mu]}$. Then the  collapsing time
$T(\lambda) = T(\Sigma^\lambda)$ of the curves of the crown is a continuous increasing function of $\lambda\in[\zeta,\mu].$
The level lines shortening at time $t<T(\Sigma^\mu)$  transforms   $\Sigma^{[\zeta,\mu]}$ into  a closed crown
\begin{equation*}
 LLS(t)\left(\Sigma_0^{[\zeta,\mu]}\right)=\Sigma_t^{[\max(\zeta, T^{-1}(t)),\mu]}.
\end{equation*}
\end{prop}

\begin{proof}
Since the composition of two continuous maps is continuous, and the composition of two strictly monotone maps is strictly monotone, this is an immediate consequence of Theorem \ref{propsofLLSoncurves} and of the definition of crowns. By property $(P1)$ level lines shortening preserves space-time continuity, whereas by $(P4)$ preserves strict monotonicity.

\end{proof}

In short, a crown remains a crown by level lines shortening, and is made of all curves of the initial crown which have not collapsed yet.
It is convenient to also  define the evolution of a flat region.
\begin{definition}\label{evolutionofregion}
Let $F_0^{\lambda_j,l}, \; l\in J_j$ be a flat region of $u_0\in\mathcal{VS}(\Omega)$ at level $\lambda_j$, of exterior curve $\Sigma_0^e$ and interior curves $\Sigma_0^m$, $m\in M$.
We call the \textbf{level lines shortening of the flat region $F_0^{\lambda_j,l}, \; l\in J_j$} the set defined by
\begin{equation*}
LLS(t)(F_0^{\lambda_j,l})=F_t^{\lambda_j,l} :=Int(\Sigma_t^e)\cap\left(\bigcap_{m}Ext(\Sigma_t^m)\right), \
\forall t<T(\Sigma^e)
\end{equation*}
where $\Sigma_t^e$ and $\Sigma_t^m$, $m\in M$ are the curve shortening evolutions of $\Sigma_0^e$ and $\Sigma_0^m$, $m\in M$.
\end{definition}

The initial flat region remains a region whose boundary is made of all Jordan curves of its initial boundary which have not collapsed. By the inclusion principle the last curve to disappear is the external boundary. When it collapses, the  region disappears.

\begin{figure}
\centering
\subfigure[Image associated to a very simple function and its corresponding level lines.]{
\includegraphics[width=0.4\columnwidth]{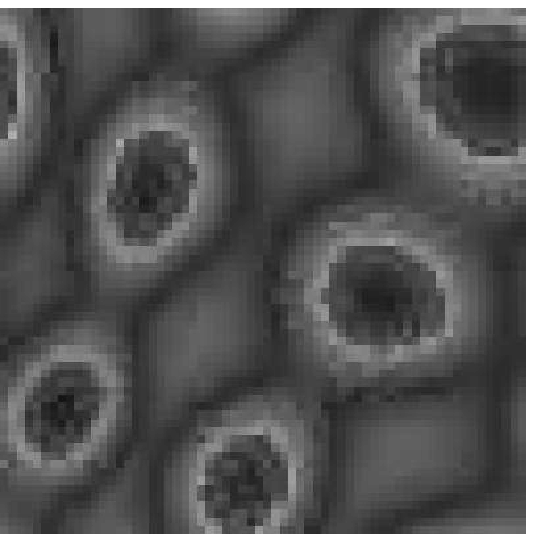}
\includegraphics[width=0.4\columnwidth]{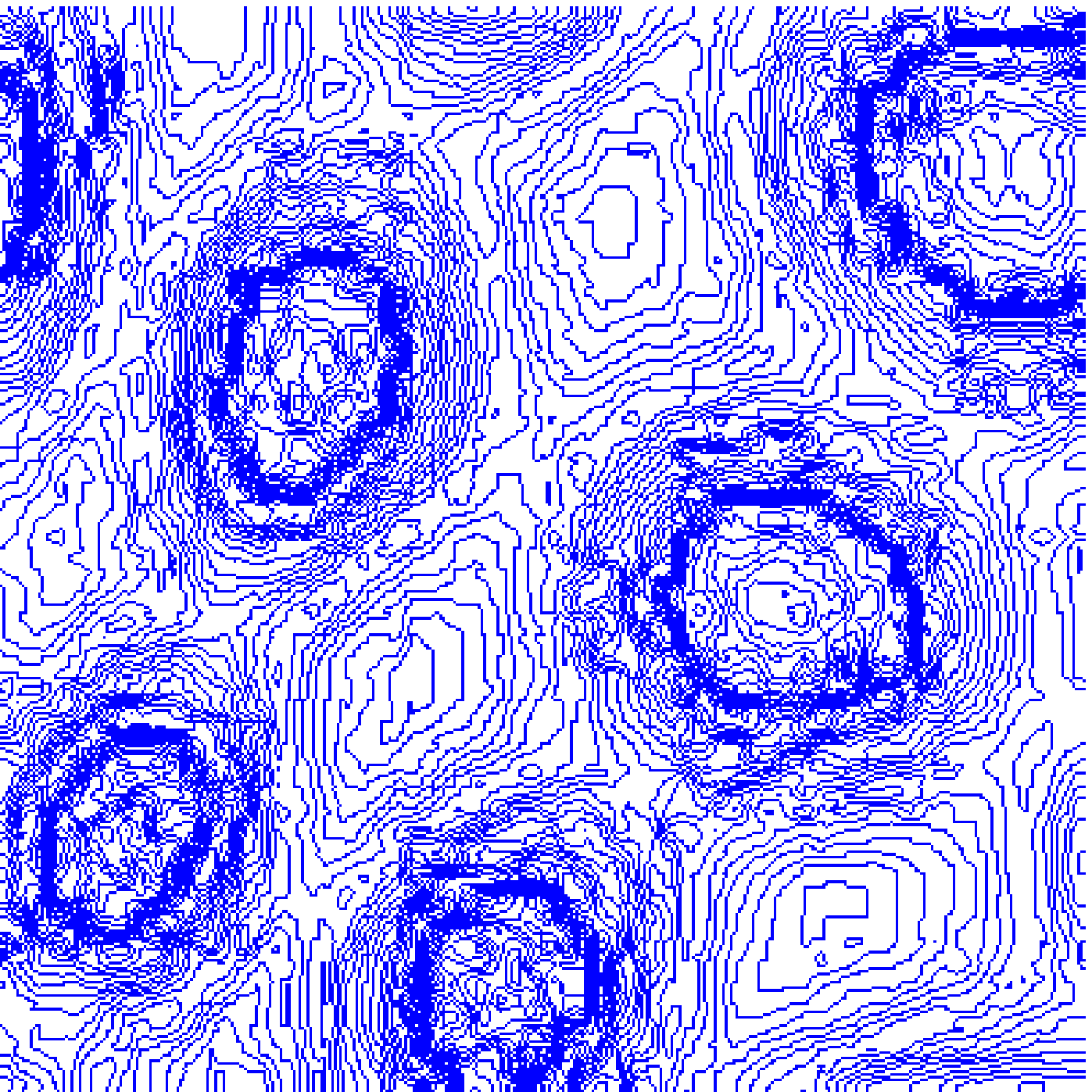}
}
\subfigure[The LLS operator maps very simple functions onto very simple functions.]{
\includegraphics[width=0.4\columnwidth]{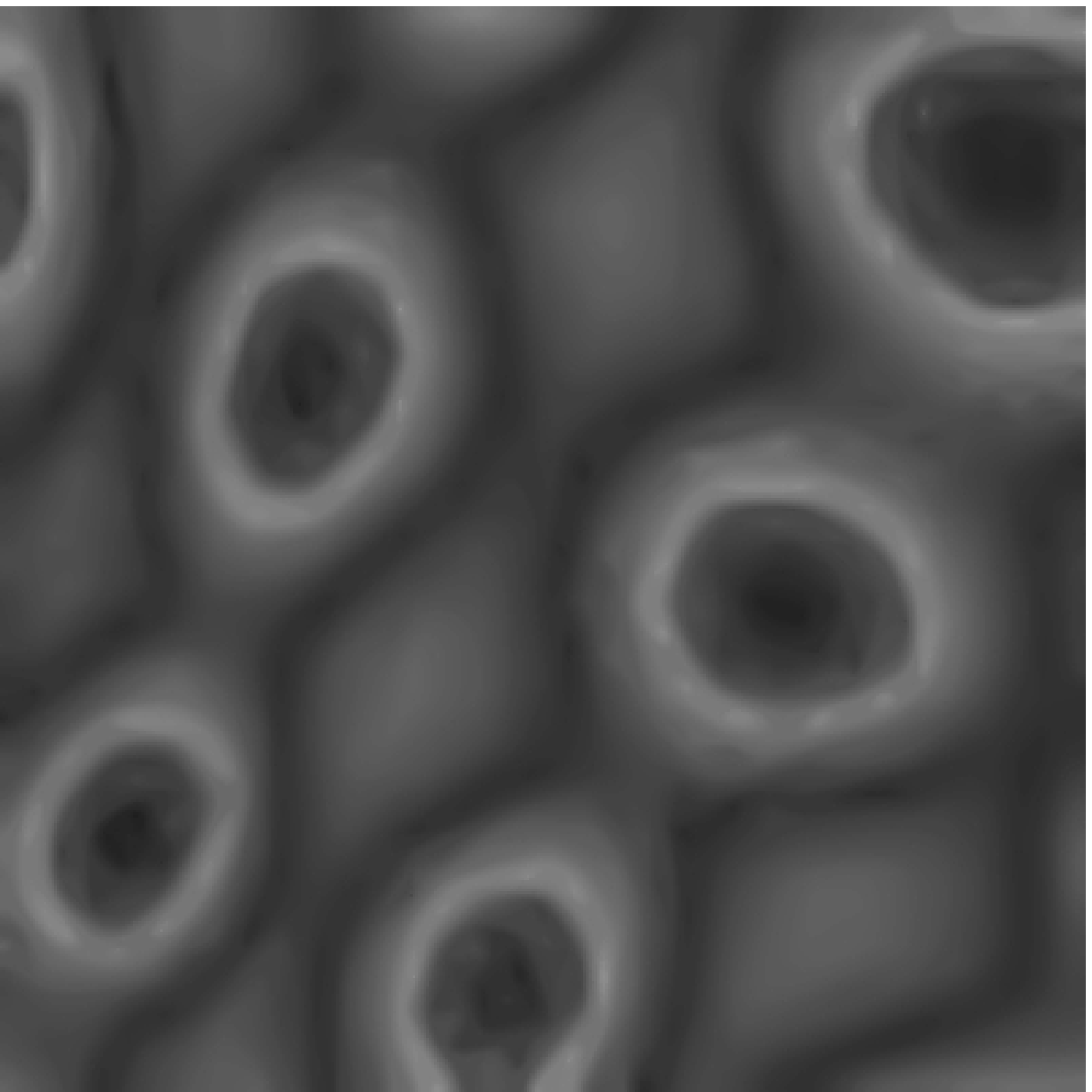}
\includegraphics[width=0.4\columnwidth]{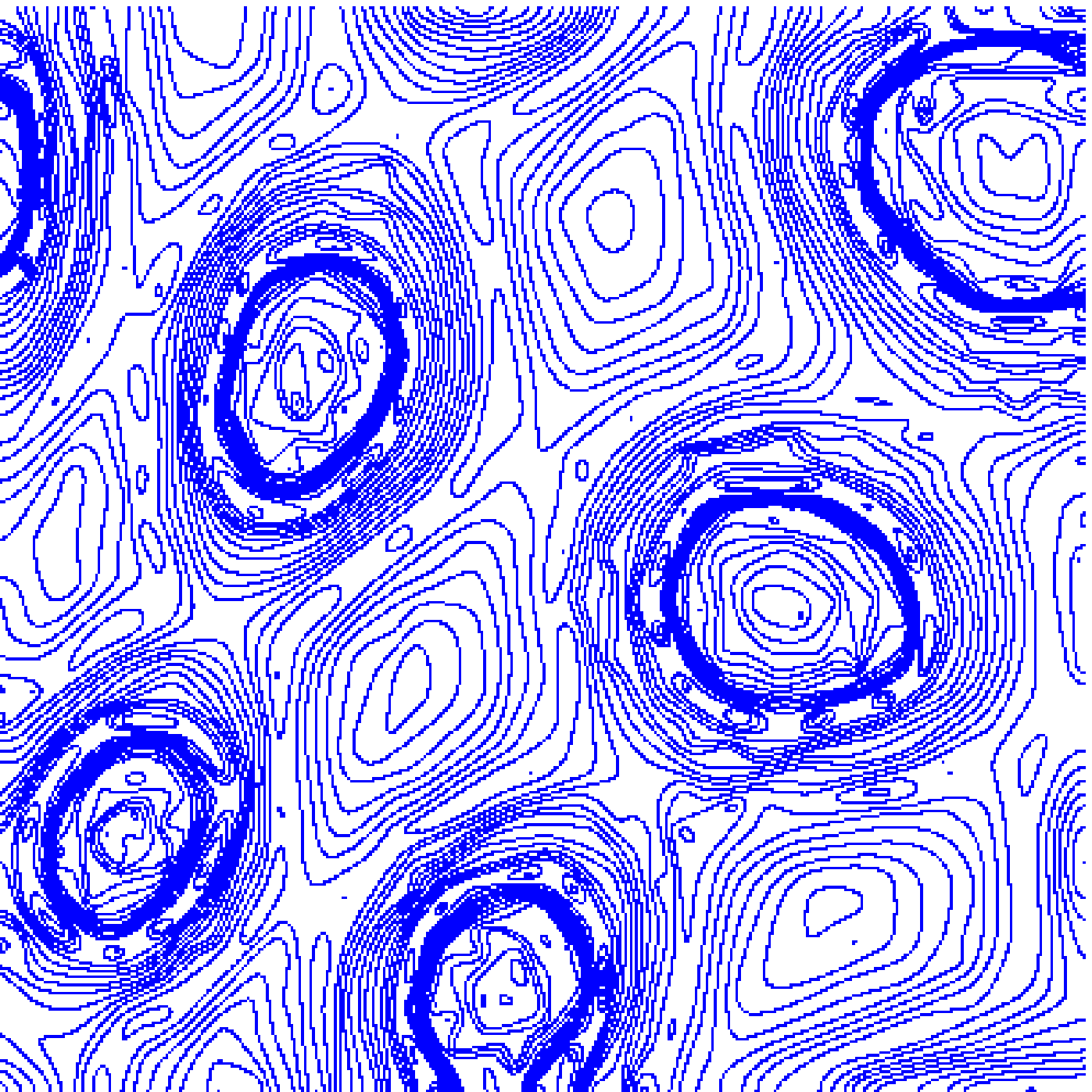}
}
\caption{\small The Level Lines Shortening (LLS) numerical chain: from the original image (top-left),  the family of all its simple level lines is extracted at quantized non-degenerate levels (top-right). Simultaneously and independently, each   level line is evolved (bottom-right)   and the evolved image having these level lines is  reconstructed  (bottom-left). Both the images and their families of level lines satisfy the topological properties $(L1)$, $(L2)$, $(L3)$. Performing the LLS evolution, digitization artifacts due to noise, compression and under-sampling are attenuated.}
\end{figure}

\begin{theorem}[Definition of LLS  for very simple functions]\label{thm::def::LLS}
Let $u_0$ be a very simple Lipschitz function, with critical levels $\{\lambda_k\}_{k=1,\cdots, n}$, Jordan curves $\Sigma_0^{\lambda, i}$ indexed by their level $\lambda$ and $i\in I_\lambda$, and of flat regions $F_0^{\lambda_j, l}$ at critical levels $\lambda_j$, indexed by $l\in J_j$.
The \textbf{level lines shortening evolution of the function $u_0$} is the function $LLS(t)(u_0)=u(\cdot,t)$  defined by
\begin{equation}
 u(x,t) = \left\{
 \begin{array}{ll}
 	 \lambda, & \hbox{ if } x \in \Sigma_t^{\lambda,i}\\
  	\lambda_j, & \hbox{ if } x \in F_t^{\lambda_j, l}\\
  	0,	& \hbox{ if } x\in \Omega\setminus\Omega_t
 \end{array}
 \right.
\end{equation}
where $\Omega_t$ is the domain surrounded by the curve shortening evolution $(\partial \Omega)_t$ of the domain boundary $\partial \Omega_0$ (which is the only zero-level curve of $u_0$).
Then this definition is complete, consistent, the evolved function is a very simple function $u(\cdot,t)\in\mathcal{VS}(\Omega)$ whose Lipschitz constant is smaller than or equal to the initial one.
\end{theorem}

\begin{proof} The initial domain $\Omega$ is partitioned in crowns and flat regions whose boundaries are either interior or exterior curves of crowns, or $\partial \Omega$.   By the min-distance property $(P5)$ in Theorem \ref{propsofLLSoncurves}, when time increases the evolved level curves of $u_0$ fall apart from each other and so do the boundary curves of the flat regions. Thus, by  Proposition \ref{propsofLLSoncurves} and Definition \ref{evolutionofregion} the crowns never meet and the flat regions are at all times the connected components of the complement in $\Omega_t$ of the union of crowns.  In other terms the evolved crowns and evolved flat regions form a partition of $\Omega_t$ given by
\begin{equation}\label{eq:partition_t}
	\Omega_t  =
		\left(\bigcup_{\substack{ k \in K}}\Sigma_t^{[\mu_k,\nu_k],k}
		\right)\cup\left(\bigcup_{\substack{j\in\{1,\cdots, n\},\\l\in J_j}} F_t^{\lambda_j,l}\right)
\end{equation}
On the other hand, the boundary of $\Omega$ is convex and remains convex by curve shortening (property $(P6)$ in Theorem \ref{propsofLLSoncurves}). Hence if points $x$ initially belonging to $\Omega$ have been crossed by the evolving boundary, we have $u(x,t)=0$.
This renders the definition complete and consistent.

Let us show that $u(\cdot,t)$ is Lipschitz.
A function is $L$-Lipschitz if and only if the min-distance between any two of its level sets with levels $\lambda$ and $\mu$ is larger than $\frac{|\lambda-\mu|}{L}$.
More precisely, since the min distance between a flat region and another, or between a flat region and a level line, is always attained on the level lines which bound the flat regions,  we have
$$ Lip(u_0) = \max\{\frac{|\zeta-\mu|}{dist(\Sigma_0^{\zeta},\Sigma_0^{\mu})}\}. $$
Since all min-distances between level lines increase (property $(P5)$, Theorem \ref{propsofLLSoncurves}), we have
$$
\frac{|\zeta-\mu|}{dist(\Sigma_0^{\zeta},\Sigma_0^{\mu})} \geq
\frac{|\zeta-\mu|}{dist(\Sigma_t^{\zeta},\Sigma_t^{\mu})}
$$
and hence the Lipschitz constant of $u(\cdot,t)$ is smaller than or equal to the  Lipschitz constant of $u_0$.
\end{proof}

\begin{definition}
We call   \textbf{Level Lines Shortening operator}, shortly LLS, the above operator acting on the class of very simple functions,
\begin{eqnarray*}
LLS(t) : \mathcal{VS}(\Omega) & \mapsto &\mathcal{VS}(\Omega) \\
	  u_0 &\mapsto & u(\cdot,t).
\end{eqnarray*}
Since the curve shortening itself is a semigroup,  $LLS(t)$ also  is a semigroup, namely $$LLS(t+s)u_0=LLS(t)(LLS(s)u_0).$$

\end{definition}

\begin{corollary}
The \textbf{level lines affine shortening evolution} of a very simple  function 
\begin{equation}
 LLAS(t)(u_0)(x): = \left\{
 \begin{array}{ll}
  \lambda, & \hbox{ if } x \in \Sigma_t^{\lambda,i}\\
  \lambda_j, & \hbox{ if } x \in F_t^{\lambda_j, l}\\
  0,	& \hbox{ if } x\in \Omega\setminus\Omega_t
 \end{array}
 \right.
\end{equation}
where $\Sigma_t^{\lambda,i}$ and $F_t^{\lambda_j, l}$ are the affine shortening evolutions of the initial level lines, respectively flat areas, is well defined and maps $\mathcal{VS}(\Omega)$ onto itself, preserving the semigroup property.
\end{corollary}

\begin{proof}
This comes immediately from Theorem \ref{propsofLLASoncurves}, which ensures that all the topological properties of curve shortening hold as well for affine shortening.
\end{proof}

\bigskip
\medskip\subsection{Properties of Level Lines Shortening}

\begin{lemma}\label{lemma::radial}
Let $\phi_0$ be a radial increasing function centered at $x_0$, i.e. $\phi_0(x)=\varphi(|x|)$ with $\varphi$ increasing. Then its level lines shortening evolution is given by
$$ \phi(x,t) = \varphi\left(\sqrt{|x|^2+2(t-t_0)}\right) $$
and its affine shortening evolution is given by
$$ \phi(x,t) = \varphi\left(\Big(|x|^{\frac43}+\frac43(t-t_0)\Big)^{3/4}\right). $$
\end{lemma}

\begin{lemma} [Local comparison] \label{lemma::localcomp}
Let $u_0$  be a very simple function, $\phi_0$ a radial increasing function centered at $x_0$, and denote by $u(\cdot,t)$ and $\phi(\cdot,t)$ their LLS/LLAS evolutions. If in a local neighborhood $\mathcal{N}(x_0)$ of $x_0$ the following holds
$$
u_0(x)\leq \phi_0(x), \forall x\in \mathcal{N}(x_0)
$$
then there exists a short time $t_0>0$ such that
$$
u(x_0,t) \leq \varphi(x_0,t), \forall 0<t<t_0.
$$\end{lemma}

\begin{proof} 

We argue for LLS, the case LLAS being analogue.
The level set $\{\phi_0< \lambda\}$ is the open disk with radius $\phi_0^{-1}(\lambda)$ and satisfies the inclusion
$$
\{ x\in \mathcal{N}(x_0); \phi_0(x)< \lambda\}  \subset
\{ x\in \mathcal{N}(x_0);  u_0(x)< \lambda\}
$$
Then every level line $\Sigma^{\lambda,\phi_0}$ of $\phi_0$ surrounds no point belonging to a level set of the same level of $u_0$. By the inclusion principle and the topological structure of the level lines for a very simple image, this property is preserved for all $t\leq t_0$ where $t_0$ is the vanishing time for the largest level line of $\phi$ in the neighborhood $\mathcal{N}(x_0)$. Indeed if the level set of $u(x,t)$  is a flat part bounded by Jordan curves, by the min-distance property $(P5)$ these Jordan curves never cross in their evolution the circles corresponding to the level set for  $\phi(x,t)$. If the level set of $u(x,t_0)$ is a finite set of level lines, in the same way the evolved level lines of $u(x,t)$ never cross the circular level line of $\phi(x,t)$.

On the other hand $x_0$ belongs to all the level lines of $\phi(\cdot,t)$, for all times $t<t_0$. Hence the value of $u(\cdot,t)$ at $x_0$ must necessarily be less than the minimum level of $\phi(\cdot,t)$, which is attained exactly at $x_0$. Consequently
$$ u(x_0,t) \leq \phi(x_0,t), \forall t<t_0. $$
\end{proof}

\begin{prop}[Space-Time Continuity]
Let $u_0\in\mathcal{VS}(\Omega)$ be $L$-Lipschitz continuous and consider its level lines (affine) shortening evolution $u(\cdot,t)= LL(A)S(t)(u_0)$, for all $t\in(0,\infty)$. Then $u\in C^0(\Omega\times[0,\infty))$.
\end{prop}

\begin{proof}
We want to find a Lipschitz type estimate for the function $u$ and we argue separately in space and time:
$$
|u(x,t) - u(x_0,t_0)| \leq |u(x,t) - u(x_0,t)| + |u(x_0,t) - u(x_0,t_0)|.
$$
By Theorem \ref{thm::def::LLS} the LLS  evolution $u(\cdot,t)$ at any time $t>0$ of the initial function $u_0$ remains $L-Lipschitz$ continuous and hence
$$
|u(x,t) - u(x_0,t)|\leq L |x-x_0|.
$$
The time-continuity follows from comparisons with shrinking cones.
More precisely, the Lipschitz continuity at time $t_0$ tells us there exists $\phi(\cdot,t_0)$ a $L$-Lipschitz radial upper barrier
$u(x,t_0) \leq \phi(x,t_0)$, touching $u$ at point $x_0$,given by
$$
\phi(x,t_0) = u(x_0,t_0)+L|x-x_0|.
$$
By Lemma \ref{lemma::radial}, its LLS evolution is
$$
\phi(x,t)=u(x_0,t_0) + L\sqrt{|x-x_0|^2+2(t-t_0)}.
$$
It follows from the local comparison with radial functions given in Lemma \ref{lemma::localcomp} that
$$
u(x_0,t)\leq \phi(x_0,t) = u(x_0,t_0) +  L\sqrt{2(t-t_0)}.
$$
Consequently
$$
|u(x,t) - u(x_0,t_0)| \leq L\left(|x-x_0| + \sqrt{2(t-t_0)}\right).
$$
This proves that $u(x,t)$ is  uniformly continuous in $t$ and $x$. Similar results hold for LLAS.
\end{proof}

The structure of  very simple functions implies that there is only a finite number of possible \textbf{collapsing times}, namely times $t_*^k>0$ where some crown collapses to a point.  At these times $t_*^1<t_*^2 \dots < t_*^m$ there is one (or several) collapse pairs $(t_*^k, \Sigma^{\lambda_k,i_k})$. The next lemma gives a stability property for flat regions.

\begin{lemma}[Flatness]\label{lemma::flatness}
Let $u_0$ be a very simple function and $u(\cdot,t)$ its LL(A)S evolution at time $t$.
Let $x_0$ be a point in a flat region of $u(\cdot,t_0)$ and suppose that it is not a collapsing point.
Then there exists $\delta_0>0$  such that $x_0$ stays in a flat region  of $u(\cdot, t)$, for all $|t-t_0|<\delta_0$.
\end{lemma}

\begin{proof}
Since $x_0$ belongs to a flat area of $u(\cdot,t_0)$ there exists a small ball $B(x_0,r_0)$ centered at $x_0$ meeting no other level line of $u(\cdot,t_0)$:
\begin{equation}\label{eq::flat}
\Sigma_{t_0}^{\lambda,i} \cap B(x_0,r_0) = \emptyset.
\end{equation}
The number of collapsing points being finite, we can also choose $r_0$ small enough, so that $B(x_0,r_0)$ contains no collapsing point of the evolution of $u$.
Let $\partial B(x_0,r(t))$ be the circle centered at $x_0$ and evolving by Curve Shortening such that $r(t_0) = r_0$.
\medskip

1. Fix $\delta_1 = r_0^2/4$ such that $\partial B(x_0,r(t))$ has not collapsed at time $t=t_0+\delta_1$. Then it follows from (\ref{eq::flat}) and from the inclusion principle that for all $t< t_0 + \delta_1$ no level line of $u(\cdot,t)$ meets the ball $B(x_0,r(t))$
$$
\Sigma_{t}^{\lambda,i} \cap B(x_0,r(t)) = \emptyset,
$$
where $r(t) = \sqrt{r_0^2 - 2(t-t_0)}$.
\medskip

2. We now prove that there exists a time $\delta_2$ and a radius $\rho >0$ such that for all $t\in(t_0-\delta_2,t_0)$ no level curve of $u(\cdot,t)$ meets $B(x_0,\rho)$.
Assume by contradiction that there exist $t_j$, $x_j$ and curves $\Sigma_{t_j}^{\lambda_j}$ of $LLS$ evolution of $u$  such that
\begin{equation}\label{eq:bwd_flat}
t_j \rightarrow t_0, \space\ x_j \rightarrow x_0 \hbox{ with } x_j\in\Sigma_{t_j}^{\lambda_j}.
\end{equation}
Consider the corresponding initial curves $\Sigma_{0}^{\lambda_j}$. Since the ball $B(x_0,r_0)$ contains no collapsing points of $u$, the evolutions $\Sigma_{t}^{\lambda_j}$ will have uniformly bounded curvature, and thus stay uniformly bounded in the $C_p^{0,1}(S^1)$ topology. It is then possible to extract a converging subsequence to $\Sigma_{0}^{\lambda}$. By the continuity property $(P2)$ in Theorem \ref{propsofLLSoncurves}, it follows that $\Sigma_{t_0}^{\lambda_0}$ contains $x_0$, which contradicts our  initial assumption.

\end{proof}

\bigskip\section{Equivalence with the curvature motions}

\medskip\subsection{Mean Curvature Motion}
We  rewrite the geometric curve shortening in the form
\begin{equation}\tag{CS}
\frac{\partial x}{\partial t}=k\nu,
\end{equation}
where $k$ denotes  the scalar curvature and $\nu$  the unit normal to the curve, with $\nu$ continuous and the sign of $k$ guaranteeing that $k\nu$ points towards the interior of the domain surrounded by the curve at convex points and towards the exterior at concave points.  We are interested in its equivalence with
\begin{equation}\tag{MCM}
\left\{
  \begin{array}{ll}
    \medskip u_t= |Du| curv(u),& \hbox{ in } \R^2\times [0,\infty)\\
    u(\cdot,0)=u_0,& \hbox{ on } \R^2.
  \end{array}
\right.
\end{equation}
where
$$
|Du|curv(u)=|Du|div(\frac{Du}{|Du|})=\sum_{i,j=1}^2(\delta_{ij}-\frac{u_{x_i}u_{x_j}}{|Du|^2})u_{x_ix_j}.
$$
We refer to a viscosity solution for the parabolic PDE, which is defined in terms of point-wise behavior with respect to a smooth test function. We use herein the definition presented by Morel and Guichard in \cite{guichard00image}, which was proven by Barles and Georgelin \cite{BG95} to be equivalent with the viscosity solutions given by Evans and Spruck in \cite{ES91} and Chen, Giga and Goto in \cite{CGG91}.

\begin{definition}\label{def::visc_sol}
A function $u\in C(\R^2\times[0,\infty))\cap L^\infty(\R^2\times[0,\infty))$ is a \textbf{viscosity sub-solution} of (MCM)
iff for each $\phi\in C^\infty(\R^2\times[0,\infty))$ such that $u-\phi$ has a local maximum at $(x_0,t_0)$ we have
\begin{eqnarray*}
\phi_t(x_0,t_0)\leq |D\phi|curv(\phi)(x_0,t_0)
&&\hbox{ if }  D\phi(x_0,t_0)\neq0 \\
\phi_t(x_0,t_0)\leq 0
&&\hbox{ if }  D\phi(x_0,t_0)=0 \hbox{ and } D^2\phi(x_0,t_0)=0.
\end{eqnarray*}
A function $u\in C(\R^2\times[0,\infty))\cap L^\infty(\R^2\times[0,\infty))$ is a \textbf{viscosity super-solution} of (MCM)
iff for each $\phi\in C^\infty(\R^2\times[0,\infty))$ such that $u-\phi$ has a local minimum at $(x_0,t_0)$ we have
\begin{eqnarray*}
\phi_t(x_0,t_0) \geq |D\phi|curv(\phi)(x_0,t_0)
&&\hbox{ if }  D\phi(x_0,t_0)\neq0 \\
\phi_t(x_0,t_0)\leq 0
&&\hbox{ if }  D\phi(x_0,t_0)=0 \hbox{ and } D^2\phi(x_0,t_0)=0.
\end{eqnarray*}
\end{definition}

\begin{remark} \label{simplificationdefinitionviscosity}The above definition can be further simplified \cite{guichard00image}:
replacing ``local maximum (minimum)`` with ''strict local maximum (minimum)`` one obtains an equivalent definition of viscosity solutions. Furthermore, it is enough to consider test functions of the form $\phi(x,t)=f(x) + g(t).$
\end{remark}

\begin{theorem}\label{thm::equivLLS}
Let $u_0\in\mathcal{VS}(\Omega)$.
Then the Level Lines Shortening evolution of the function $u_0$,
$$u(x,t)=LLS(t)u_0(x), \forall x\in\R^2, \forall t\in [0,\infty)$$
is a viscosity solution for  (MCM) with  initial data $u_0$,
\end{theorem}

\begin{proof}
It is sufficient to check that $u(x,t)$ is a viscosity sub-solution. Analogous assertions hold for viscosity super-solutions. Let $\phi\in C^\infty(\R^2\times[0,\infty))$ such that $u-\phi$ has a strict local maximum at $(x_0,t_0)$. Adding if necessary a constant, suppose that
\begin{equation}\label{eq:max_cond}
\left\{
  \begin{array}{ll}
  u(x_0,t_0)  =  \phi(x_0,t_0)=\lambda \\
  u(x,t) <  \phi(x,t),   \forall (x,t)\in V
  \end{array}
\right.
\end{equation}
where $V$ is a small neighborhood of $(x_0,t_0)$. The proof is completed by the next three lemmas, where we  distinguish two situations: either the point $x_0$ is inside a flat region of $u(\cdot,t_0)$ (lemma \ref{onaflatpoint}), or it belongs to some level line, singular or not of this function (lemma
\ref{onalevellinewithgradientzero} in the case where $D\phi(t_0,x_0)=0$,   and lemma \ref{onalevelinewithnonnullgradient} when the gradient is not zero).
\end{proof}

\begin{lemma} \label{onaflatpoint}
 Let $x_0$ be a point in a flat area of $u(\cdot,t_0)$. Then
 $$D\phi(x_0,t_0)=0 \hbox{ and }\phi_t(x_0,t_0)=0.$$
\end{lemma}
\begin{proof}
By   Lemma \ref{lemma::flatness}, the function $u$ is constant in a small neighborhood $\mathcal N(x_0, t_0)$.
From the local maximum condition we deduce that the point $(x_0,t_0)$ is a local minimum for the test
function $\phi\in C^{\infty}$, which yields the conclusion.

\end{proof}

\noindent We consider now the case when $x_0$ belongs to a level line $\Sigma_{t_0}^{\lambda,u}$ of the function $u(\cdot,t_0)$. By the construction of   $LLS(t)u_0$ this level line is following the classical curve shortening.

\begin{lemma}  \label{onalevelinewithnonnullgradient}
Let $x_0$ belong to a level line $\Sigma_{t_0}^{\lambda,u}$ of the function $u(\cdot,t_0)$. Let $\phi $ be a smooth test function such that at the maximum point $(x_0,t_0)$ of $u-\phi$
$$
D\phi(x_0,t_0)\neq 0.
$$
Then $\phi$ satisfies
\begin{equation}\nonumber
\phi_t\leq |D\phi|div(\frac{D\phi}{|D\phi|})  \hbox{ at } (x_0,t_0).
\end{equation}
\end{lemma}

\begin{proof}

1. The non-degeneracy condition $D\phi(x_0,t_0)\neq 0$ and the regularity of the test function $\phi$ imply by the implicit function theorem that the iso-level set
$$
\Sigma_{t_0}^{\lambda,\phi}=\{x\in\Omega;\phi(x,t_0)=\lambda\}
$$
is a smooth graph in a neighborhood of $x_0$. A unit normal vector of $\Sigma_{t_0}^{\lambda,\phi}$ at point $x_0$ is
$$
\nu^\phi(x_0,t_0)=\frac{D\phi}{|D\phi|}(x_0,t_0).
$$
On the other hand, $x_0$ belongs to the smooth Jordan level line $\Sigma_{t_0}^{\lambda,u}$ of $u$.
By the local maximum condition at point $(x_0,t_0)$ the two graphs $\Sigma_{t_0}^{\lambda,\phi}$ and $\Sigma_{t_0}^{\lambda,u}$ are tangent at  $x_0$ and do not intersect in a small neighborhood of the point. Therefore, the unit normal vectors of these curves coincide up to their sign. We set
$$
\nu^{u}(x_0,t_0) = \frac{D\phi}{|D\phi|}(x_0,t_0).
$$
Furthermore, for short times  $t\in(t_0-\delta,t_0+\delta)$ with $\delta>0$ small enough, the $\lambda$ level set of $\phi(\cdot,t)$  denoted by
$$\Sigma_t^{\lambda,\phi}= \{x\in\Omega_t;\phi(x,t) = \lambda\}$$
remains a smooth graph in a neighborhood of $x_0$.
By the maximum condition (\ref{eq:max_cond}), $\Sigma_t^{\lambda,\phi}$ stays on the same side of   $\Sigma_t^{\lambda,u}$ (see Figure \ref{fig::ll_test}).

To fix ideas, suppose that for points $x$ close enough to $\Sigma_{t_0}^{\lambda,u}$,   $u(x)>\lambda$ in the interior domain bounded by the level line $\Sigma_{t_0}^{\lambda,u}$. Then the $\lambda-$level line of the test function $\Sigma_{t_0}^{\lambda,\phi}$ lies locally outside  the same domain. In addition, the normal vectors $\nu^u(x_0,t_0)$ and $\nu^\phi(x_0,t_0)$ point inwards the interior domain bounded by the level line $\Sigma_{t_0}^{\lambda,u}$.
\medskip

\begin{figure} \label{fig::ll_test}
   \centering
     \includegraphics[width=0.6\columnwidth]{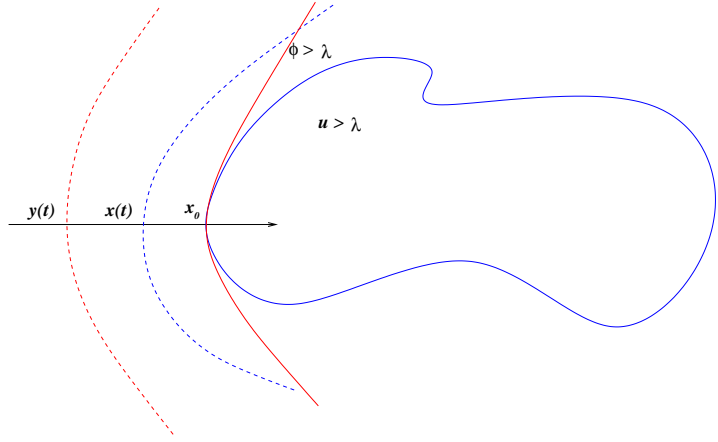}
\caption{\small The level line $\Sigma_t^{\lambda,\phi}$ (in red) stays on the same side of the level line $\Sigma_t^{\lambda,u}$ (in blue). For $t<t_0$, we consider the backwards locations $x(t)$ of the point $x_0$ on the level lines $\Sigma_t^{\lambda,u}$ as well as the intersections $y(t)$ of the normal direction with the level lines $\Sigma_t^{\lambda,\phi}$ of the test function.}
\end{figure}

2. We consider now the backwards locations $x(t)$ on the curves $\Sigma_t^{\lambda,u}$  of  the point $x_0\in\Sigma_{t_0}^{\lambda,u}$, curves which represent the curve shortening evolutions of some level line $\Sigma_0^{\lambda,u}$. Let $\nu(x(t),t)$ be the unit inward vector at point  $x_0$ of the level line $\in\Sigma_{t_0}^{\lambda,u}$. The vector points in the direction of $x_0-x(t)$ and hence there exists $d(x(t),t)$ such that
\begin{equation}\label{eq::x(t)}
x(t) = x_0 - d(x(t),t)\nu(x(t),t)
\end{equation}
By the smoothness property $(P1)$ of Theorem \ref{propsofLLSoncurves},
\begin{equation}\nonumber
\lim_{t\nearrow t_0} \nu(x(t),t) = \nu^u(x_0,t_0).
\end{equation}
Let $y(t)$ be the intersection point of the outward normal direction $-\nu(x(t),t)$ (we take here into account that we have considered upper level sets, i.e. $\nu^\phi$ points inwards) with the level curve $\Sigma_t^{\lambda,\phi}$ (situated outside $\Sigma_{t_0}^{\lambda,u}$). Then
there exists $d(y(t),t)$ such that
\begin{equation}\label{eq::y(t)}
y(t) = x_0 - d(y(t),t)\nu(x(t),t)
\end{equation}
and $D\phi(y(t),t)\neq 0$, since $\Sigma_t^{\lambda,\varphi}$ remains a local graph for short times. From the implicit function theorem we deduce that $y(t)$ is uniquely defined and varies smoothly in time. Thus the following limit exists:
$$
\lim_{t\nearrow t_0} \frac{y(t) - x_0}{t-t_0} = \frac{\partial y}{\partial t}(t_0).
$$
Furthermore, since  $\Sigma_t^{\lambda,\phi}$  stays on the same side of $\Sigma_t^{\lambda,u}$ we have
$$
d(x(t),t) \leq d(y(t),t).
$$
Taking the inner product with $\nu(x(t,)t)$ in equations (\ref{eq::x(t)}) and (\ref{eq::y(t)}) and dividing by $t<t_0$, the previous inequality  implies that
$$
\langle \frac{x(t) -x_0}{t-t_0},\nu(x(t),t)\rangle \leq  \langle \frac{y(t) -x_0}{t-t_0},\nu(x(t),t)\rangle.
$$
Passing to the limits as $t\rightarrow t_0$ we have
\begin{equation}\label{eq::xy_speeds}
\langle \frac{\partial x}{\partial t} (t_0),\nu^u(x_0,t_0)\rangle \leq \langle \frac{\partial y}{\partial t} (t_0),\nu^u(x_0,t_0)\rangle 
\end{equation}
and taking into account that (CS) gives
  \begin{equation}\label{eq::CSnormal}
  \langle \frac{\partial x}{\partial t} (t_0),\nu^u(x_0,t_0)\rangle = k^u(x_0,t_0)
  \end{equation}
we get
$$
k^u \leq \langle \frac{\partial y}{\partial t}(t_0),\nu^u(x_0,t_0)\rangle.
$$
But $\Sigma_{t_0}^u$ and $\Sigma_{t_0}^\phi$ are ordered by inclusion and meet at point $x_0$. Thus, their curvatures at $x_0$ are ordered
$ k^\phi(x_0,t_0) \leq k^u(x_0,t_0).$
On the other hand, from the regularity of the test function and the fact that $D\phi(x_0,t_0)\not = 0$ the curvature at $x_0$ can be expressed as
$$
k^\phi(x_0,t_0)= -\hbox{div}(\frac{D\phi}{|D\phi|})(x_0,t_0)
$$
Consequently
\begin{equation}\label{speed_div}
 -\hbox{div}(\frac{D\phi}{|D\phi|})(x_0,t_0)\leq  \langle \frac{\partial y}{\partial t}(t_0),\nu^u(x_0,t_0)\rangle.
\end{equation}\medskip

3. The sequence of points $y(t)$ found before belongs to the $\lambda-$level set of the test function, thus we have
$$
\phi(y(t),t)= \lambda, \hbox{ for } t\in(t_0-\delta,t_0].
$$
Differentiating this identity with respect to $t$ one gets for $t=t_0$
$$
\langle \frac{\partial y}{\partial t}(t_0),D\phi(x_0,t_0)\rangle + \phi_t(x_0,t_0)= 0.
$$
But at the touching point $x_0$ the normal vector to $\Sigma_{t_0}^{\lambda,\phi}$ can be expressed as $$\frac{D\phi}{|D\phi|}(x_0,t_0)=\nu^u(x_0,t_0).$$ Consequently the equality above at point $(x_0,t_0)$ becomes
$$
\phi_t(x_0,t_0)= -|D\phi|(x_0,t_0) \langle \frac{\partial y}{\partial t}(x_0,t_0),\nu^u(x_0,t_0)\rangle
$$
which by inequality (\ref{speed_div}) implies
$$
\phi_t(x_0,t_0)\leq |D\phi|\hbox{div}(\frac{D\phi}{|D\phi|})(x_0,t_0)\rangle.
$$
\end{proof}

\begin{remark}
If a point $x_0$ is a collapsing point at time $t_0$, then
$$u(x,t) = \lambda, \forall t\in[t_0,t_0+\delta)$$
Therefore the test function $\varphi$ satisfies $D\phi(x_0,t_0) = 0$.
\end{remark}

\begin{lemma}\label{onalevellinewithgradientzero}
Let $x_0$ belong to a level line $\Sigma_{t_0}^{\lambda,u}$ of the function $u(\cdot,t_0)$. Let $\phi $ be a test function such that
$
D\phi(x_0,t_0)= 0, \; D^2\phi(x_0,t_0)=0.
$
Then $\phi$ satisfies
$
 \phi_t(x_0,t_0)\leq 0$.
\end{lemma}

\begin{proof}
By Remark \ref{simplificationdefinitionviscosity} we can assume that the test function has the form $\phi(x,t)=f(x)+g(t)$.
By assumption, for every $(x,t)$ in a neighborhood of $(t_0,x_0)$ we have
\begin{equation}\label{localinequality}
u(x,t)-f(x)-g(t)\leq u(x_0,t_0)-f(x_0)-g(t_0).
\end{equation} 

1. Assume first that $\Sigma_{t_0}^{\lambda,u}$ is not  reduced to a point and say it is the curve shortening evolution of the original level line $\Sigma_{0}^{\lambda,u}$. Denote by  $\Sigma_{t}^{\lambda,u}$ the intermediate evolutions for $0<t<t_0$. Consider, as before, for short times  $t \in(t_0-\delta,t_0]$  the points $x(t)$ belonging to $\Sigma_{t}^{\lambda,u}$ such that $x(t) \rightarrow x_0 \hbox{ as } t\to t_0$.
Since $u(x(t),t)=u(x_0,t_0)$ from inequality \eqref{localinequality}  we get
\begin{equation}\label{inegalitefg}
g(t_0)-g(t)\leq f(x(t))-f(x_0).
\end{equation}
However, since the level line evolves by curve shortening, we also have
\begin{equation}\label{eq::CSasymp}
x(t)=x_0+(t-t_0)k(x_0,t_0)\nu(x_0,t_0)+o(t-t_0).
\end{equation}
Substituting this asymptotic expansion in \eqref{inegalitefg} and recalling that  $Df(x_0,t_0)=0$ yields
$$g'(t_0)(t_0-t)\leq o(t_0-t),$$ which implies $g'(t_0)\leq 0$.
Since $g'(t_0)=\phi_t(x_0,t_0)$, this proves the announced statement.
\medskip

2. The only case not treated by the above argument is when $x_0$ is the collapsing point of some level line $\Sigma_t^{\lambda,u}$. Accordingly, $t_0$ is its collapsing time.
In this case, the level line does not have a normal direction at $(x_0,t_0)$ and consequently equation (\ref{eq::CSasymp}) is not valid anymore.

Nevertheless, by Theorem \ref{propsofLLSoncurves}, property $(P3)$, for short times $t\in(t_0-\delta,t_0)$ the points $x(t)$ lie approximatively on a circle of radius $R(t)$ respectively, with $R(t_0)=0$.
Thus
\begin{equation}\label{eq::CSasympR}
 |x(t)-x_0|\simeq   \sqrt{2(t_0-t)}.
\end{equation}
Substituting this asymptotic expansion in \eqref{inegalitefg} and taking into account that  $$Df(x_0,t_0)=0 \hbox{ and } D^2f(x_0,t_0)=0$$ we obtain
$$g'(t_0)(t_0-t)\leq o(t_0-t),$$ which implies  again $g'(t_0)\leq 0$.
\end{proof}

The Level Lines Shortening can be extended by density to an operator acting on the class of Lipschitz functions.

\begin{corollary}
For each $u_0\in Lip(\Omega)$ consider a uniform approximation by very simple functions $u_0^n\in\mathcal{VS}(\Omega)$ and define its LLS evolution by
$$
LLS(t)u_0(x) = \lim_{n\rightarrow\infty}\left(LLS(t)u_0^n(x)\right),
$$
Then the right hand side member is well defined and is a solution of $(MCM)$ with initial data $u_0$.
\end{corollary}

\begin{proof}
From the previous theorem we know that the LLS evolutions of the very simple functions $u_0^n$
$$
u^n(\cdot,t) = LLS(t)u_0^n
$$
are solutions of the mean curvature equation. By the comparison principle we know that
$$
\max_{x\in\Omega}(u^n(x,t) - u^m(x,t)) \leq \max_{x\in\Omega}(u_0^n(x) - u_0^m(x)).
$$
Since the sequence of very simple functions is uniformly convergent, we deduce that for each $t>0$, the family $\{LLS(t)u_0^n\}_n \subset\mathcal{VS}(\Omega)$ is a Cauchy sequence in the $||\cdot||_\infty$ norm.
By the stability properties of viscosity solutions, the limit function $$u(x,t) = \lim_{n\rightarrow\infty}{u^n(x,t)}$$
is also a viscosity solution of $(MCM)$ and hence satisfies
$$
Lip(u(\cdot,t))\leq Lip(u_0).
$$
\end{proof}

\medskip \subsection{Affine Curvature Motion}
Similarly, one can connect the affine shortening
\begin{equation}\tag{AS}
\frac{\partial x}{\partial t}=k^{1/3}\nu,
\end{equation}
where $\nu(\cdot,t)$ is the inner unit normal vector of the curve $x(\cdot,t)$ and $k(\cdot,t)$ the signed scalar curvature corresponding the choice of $\nu(\cdot,t)$,
with the affine curvature motion
\begin{equation}\tag{ACM}
\left\{
  \begin{array}{ll}
    \medskip u_t= |Du| \big(curv(u)\big)^{1/3},& \hbox{ in } \R^2\times [0,\infty)\\
    u(\cdot,0)=u_0,& \hbox{ on } \R^2.
  \end{array}
\right.
\end{equation}
In the definition of these nonlinear evolutions, for $x\in\R$, $x^{1/3}$ stands for $\hbox{sgn}(x)|x|^{1/3}$.
\begin{theorem}\label{thm::equivLLAS}
Let $u_0\in\mathcal{VS}(\Omega)$.
Then the Level Lines Affine Shortening evolution of the function $u_0$,
$$u(x,t)=LLAS(t)u_0(x), \forall x\in\R^2, \forall t\in [0,\infty)$$
is a viscosity solution for  (ACM) with  initial data $u_0$.
\end{theorem}

\begin{proof}
The proof of Theorem \ref{thm::equivLLS} is purely geometric, thus the same arguments apply for level lines affine shortening. 
\medskip

1. When $D\phi(x_0,t_0)\neq 0$ we need to estimate from above the right hand side of 
$$
\phi_t(x_0,t_0)= -|D\phi|(x_0,t_0) \langle \frac{\partial y}{\partial t}(x_0,t_0),\nu^u(x_0,t_0)\rangle.
$$
The proof is  literally the same up to inequality (\ref{eq::xy_speeds})
$$
\langle \frac{\partial x}{\partial t} (t_0),\nu^u(x_0,t_0)\rangle \leq  \langle \frac{\partial y}{\partial t} (t_0),\nu^u(x_0,t_0)\rangle. 
$$
The only difference it makes with the previous proof is when (AS) comes into play. More precisely the evolution equation (\ref{eq::CSnormal}) at the maximum point $(x_0,t_0)$ should be replaced by an affine shortening evolution 
\begin{equation}\label{eq::ASnormal}
\langle \frac{\partial x}{\partial t} (t_0),\nu^u(x_0,t_0)\rangle = \left(k^u(x_0,t_0)\right)^{1/3}
\end{equation}
On the other hand $k^\phi(x_0,t_0)\leq k^u(x_0,t_0)$ which implies 
$$
\big(k^\phi(x_0,t_0)\big)^{1/3} = \hbox{sgn}(k^\phi)|k^\phi(x_0,t_0)|^{1/3} 
\leq 
\hbox{sgn}(k^u)|k^u(x_0,t_0)|^{1/3} = \big(k^u(x_0,t_0)\big)^{1/3}
$$
But for the test function $\phi$, $k^\phi(x_0,t_0)= -curv(\phi)(x_0,t_0)$. Hence 
$$
-(curv(\phi))^{1/3} \leq  \langle \frac{\partial y}{\partial t} (t_0),\nu^u(x_0,t_0)\rangle. 
$$
from where we deduce the desired viscosity inequality.
\medskip

2. For the case $D\phi(x_0,t_0)= 0$ and  $D^2\phi(x_0,t_0)= 0$
it is sufficient to replace the asymptotic expansions (\ref{eq::CSasymp}) by
\begin{equation}\label{eq::ASasymp}
x(t)=x_0+(t-t_0)\big(k(x_0,t_0)\big)^{1/3}\nu(x_0,t_0)+o(t-t_0).
\end{equation}
respectively (\ref{eq::CSasympR}) by
\begin{equation}\label{eq::ASasympR}
 |x(t)-x_0|\simeq   \sqrt[\frac43]{\frac43(t_0-t)}.
\end{equation}
This concludes the proof.

\end{proof}

\bigskip\section{Numerical Implementation and Applications}

The discrete Level Lines Shortening Algorithm performs accurate sub-pixel evolution by mean curvature motion.  
The complexity of the algorithm is directly proportional to the total variation of the image, since it acts simultaneously and independently on all of the level lines of the image (of course up to a quantization step). 
The main goal of the implementation is to obtain and move level lines with arbitrarily high sub-pixel precision, overcoming thus all the drawbacks of finite difference schemes based on pixel approximations. 
Moving simultaneously  level lines extracted with high sample precision allows straight level lines with high gradient to stand still with LLS, whereas they are diffused by FDS, even in its stack variant. Further example and details about the numerical implementation are given in  \cite{CMM11}. The following example illustrates the recovery of shapes freed from their aliasing, JPEG, and noise artifacts.

\subsection{Image restoration and visualization}
Aliasing due to pixelization is common in scanned documents.  LLAS  can be used for a graphic quality improvement smoothing  contours, see Fig. \ref{mouselaughing}.
After smoothing, pixelized level  lines become accurate curves with sub-pixel control points, whose curvature can be faithfully computed. Thus the whole chain can be viewed as a numerical preprocessing before further numerical analysis and feature extraction.  But there is also a strong interest in the direct visualization of the level lines and of the microscopic curvature map of an image. 

\begin{figure}[ht]
 \centering
 {\includegraphics[width=0.3\linewidth]{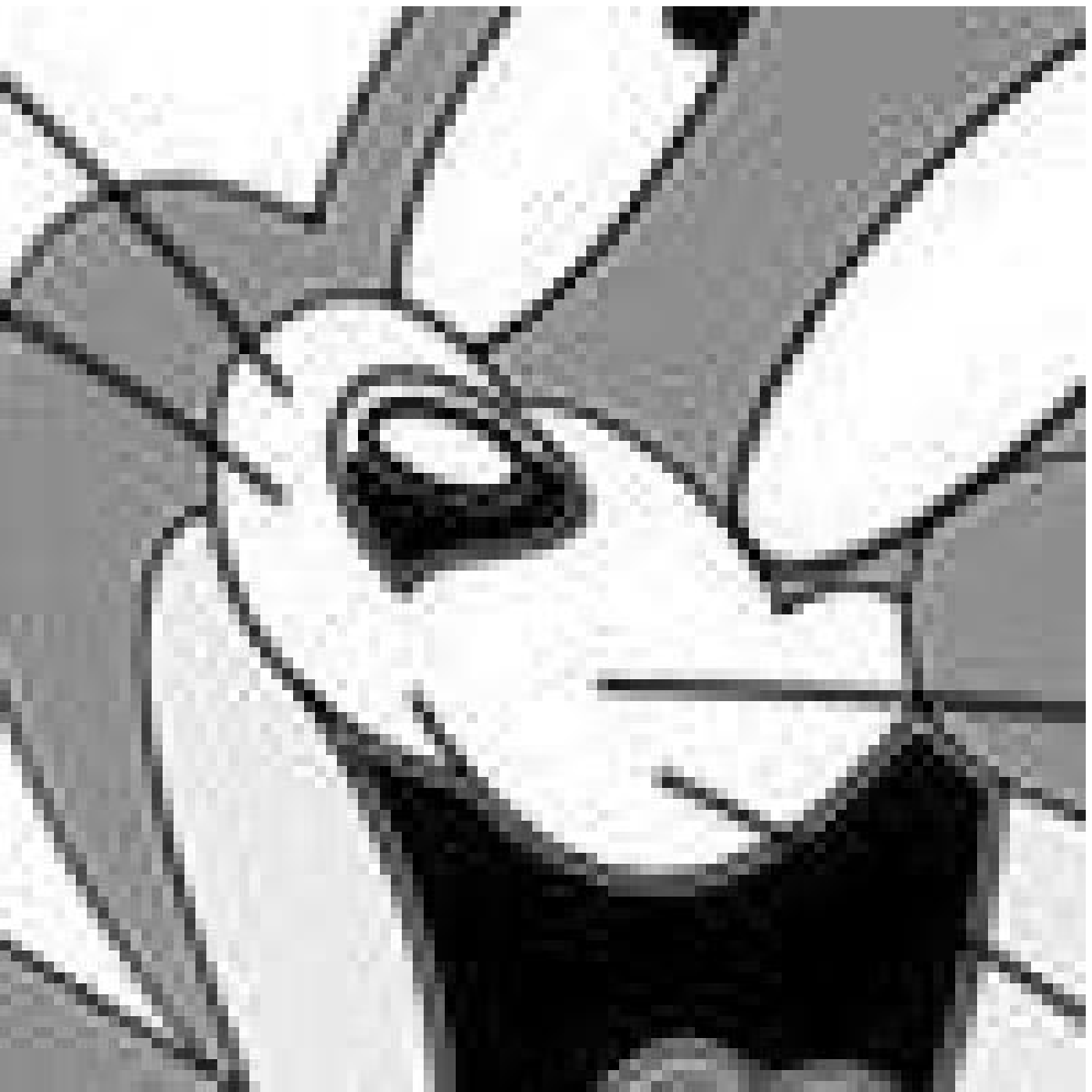}}
 {\includegraphics[width=0.3\linewidth]{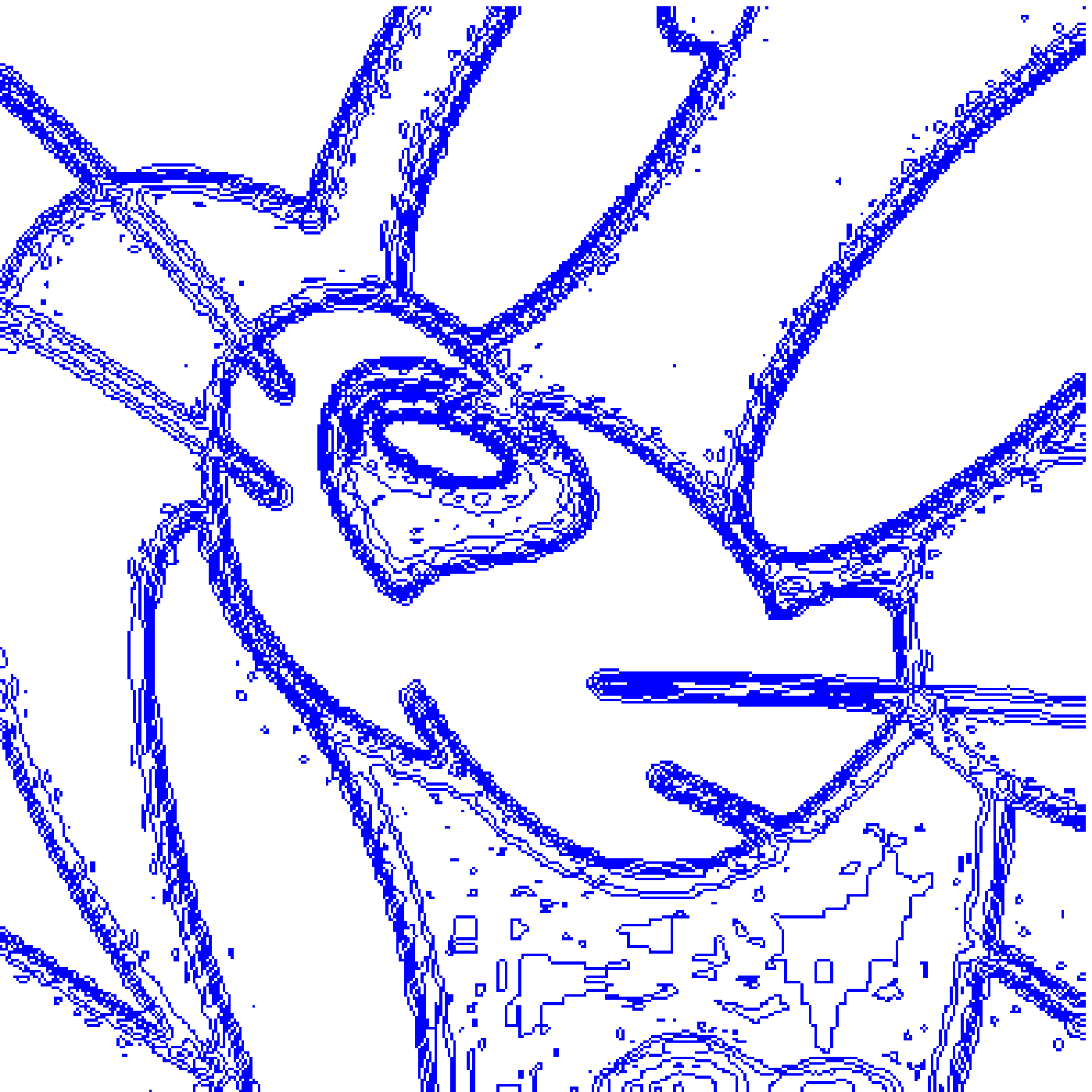}}
 
 {\includegraphics[width=0.3\linewidth]{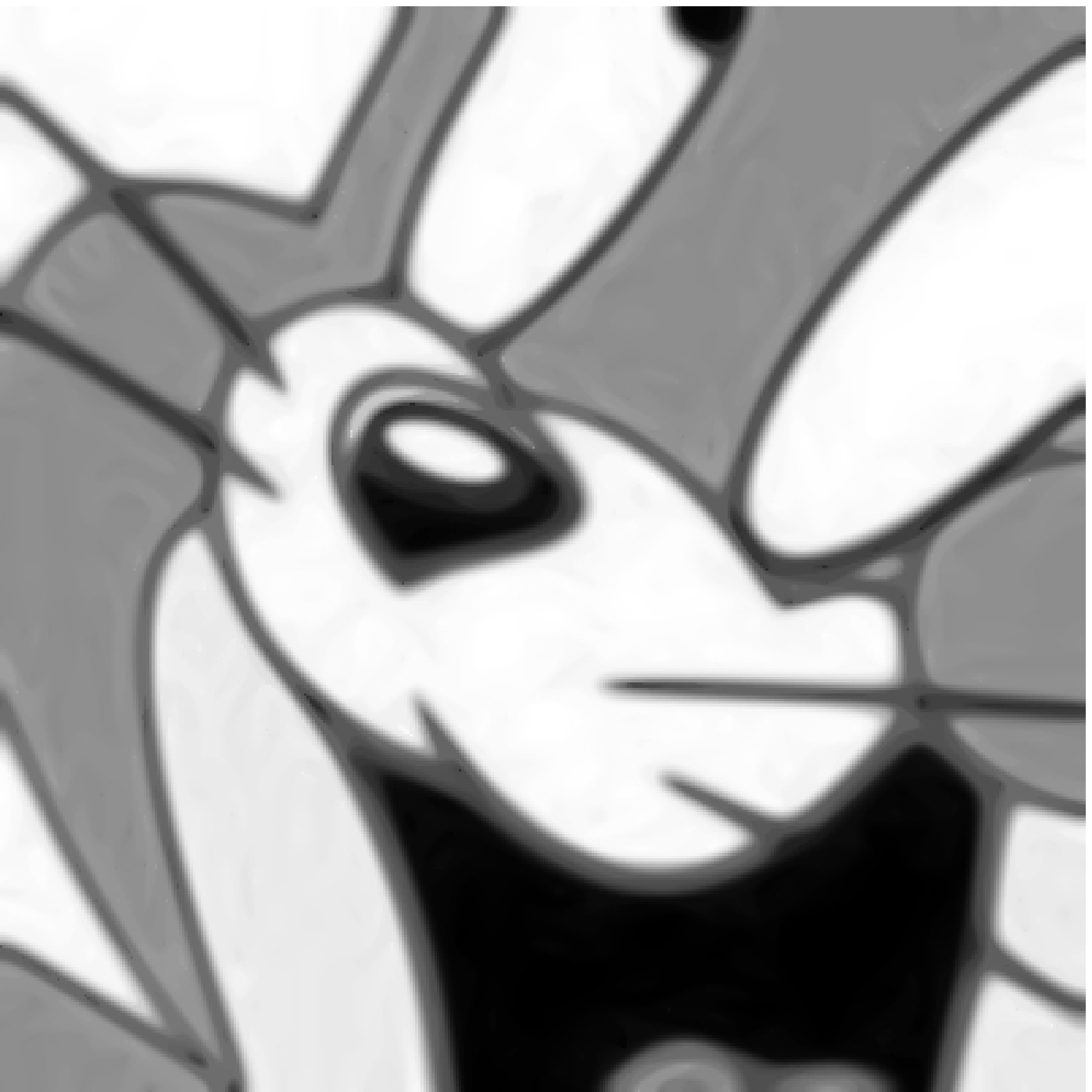}}
 {\includegraphics[width=0.3\linewidth]{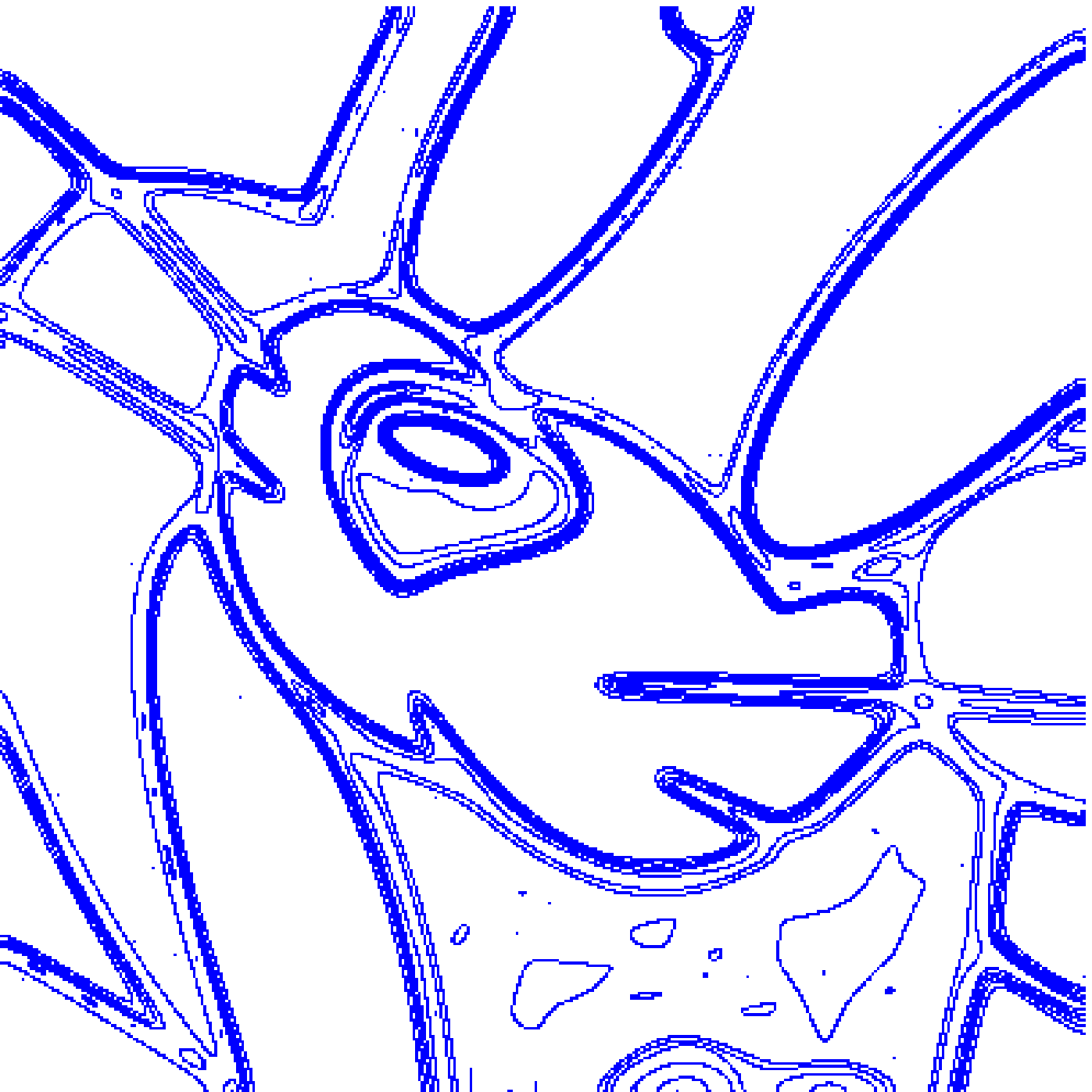}}
 \caption{\small Top: original cartoon image and its corresponding bilinear level lines. Bottom: LLAS evolution and  smoothed level lines.\label{mouselaughing}}
\end{figure}

\medskip
\subsection{Fingerprints restoration and discrimination}

Minutiae such as cores, bifurcations and ridge endings characterize uniquely fingerprints. Their detection requires a careful smoothing, particularly to avoid a spurious diffusion mixing the ridges.  The main objective of smoothing is to sieve the curvature  extrema, which allow the fingerprint discrimination.
\begin{figure}[!ht]
\centering
{\includegraphics[width=0.24\linewidth]{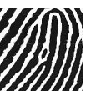}}
{\includegraphics[width=0.24\linewidth]{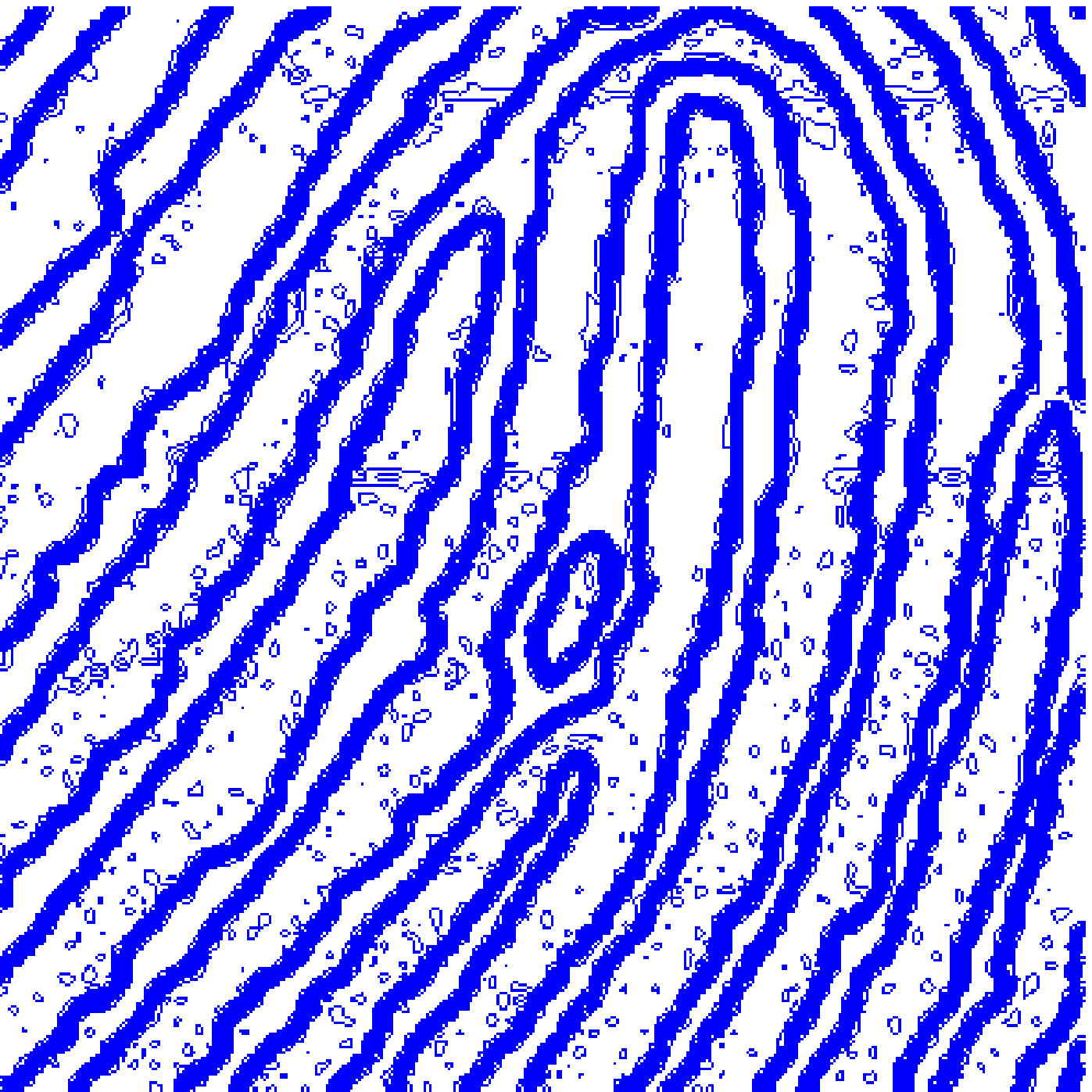}}
{\includegraphics[width=0.24\linewidth]{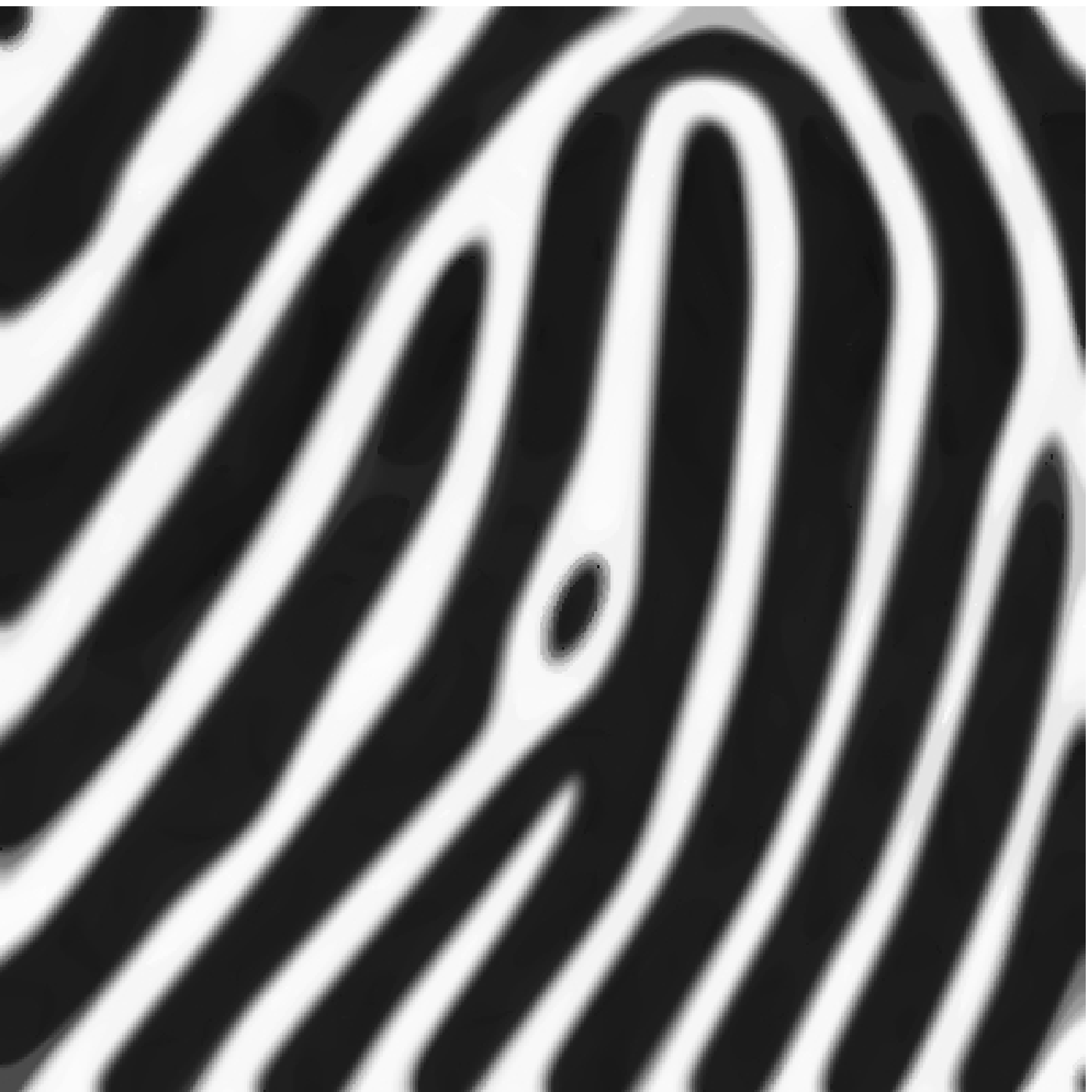}}
{\includegraphics[width=0.24\linewidth]{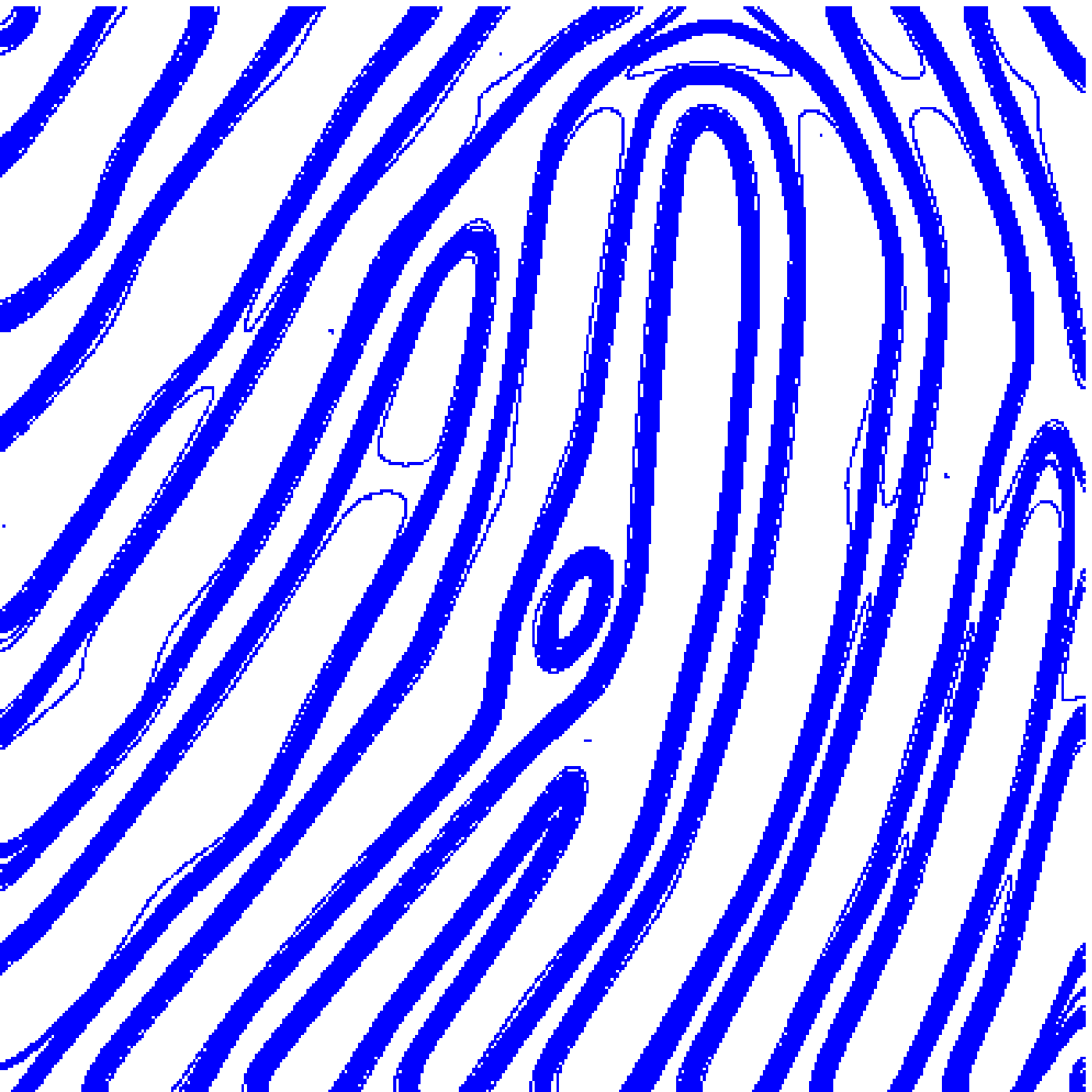}}
  \caption{\small{Original fingerprint and its level lines, Level Lines Affine Shortening of the image and the corresponding level lines.}}
  \label{FIG::fingerprint}
\end{figure}

\textbf{Acknowledgments:} This research is partially financed by the MISS project of Centre National
d'Etudes Spatiales, the Office of Naval research under grant N00014-97-1-0839 and by the European Research Council, advanced grant ``Twelve labours''.

\bibliographystyle{plain}

\end{document}